\documentclass[a4paper,11pt]{amsart}

\usepackage{fancybox}
\usepackage{amscd}
\usepackage{amsmath}
\usepackage{amssymb}
\usepackage{amsthm}
\usepackage{float}
\usepackage[dvips]{graphicx}
\usepackage{tikz}
\usepackage[all, ps, dvips]{xy}
\usepackage{color}
\usepackage[hidelinks, hyperfootnotes=false]{hyperref}
\usepackage{array}
\usepackage{amscd}

\newcommand\blfootnote[1]{%
  \begingroup
  \renewcommand\thefootnote{}\footnote{#1}%
  \addtocounter{footnote}{-1}%
  \endgroup
}

\DeclareFontFamily{U}{rsfs}{%
\skewchar\font127}
\DeclareFontShape{U}{rsfs}{m}{n}{%
<-6>rsfs5<6-8.5>rsfs7<8.5->rsfs10}{}
\DeclareSymbolFont{rsfs}{U}{rsfs}{m}{n}
\DeclareSymbolFontAlphabet
{\mathrsfs}{rsfs}
\DeclareRobustCommand*\rsfs{%
\@fontswitch\relax\mathrsfs}

\theoremstyle{plain}
\newtheorem{thm}{Theorem}[section]
\newtheorem{prop}[thm]{Proposition}
\newtheorem{lem}[thm]{Lemma}

\newtheorem{defi}[thm]{Definition}
\newtheorem{rmk}[thm]{Remark}
\newtheorem{cor}[thm]{Corollary}
\newtheorem{prop-defi}[thm]{Proposition-Definition}
\newtheorem{thm-defi}[thm]{Theorem-Definition}
\newtheorem{lem-defi}[thm]{Lemma-Definition}

\newtheorem{exam}[thm]{Example}

\newtheorem{setup-def}[thm]{Setup-Definition}
\newtheorem{notation}[thm]{Notation}

\newenvironment{sproof}{%
  \proof}{\endproof}

\newdimen\argwidth
\def\db[#1\db]{
 \setbox0=\hbox{$#1$}\argwidth=\wd0
 \setbox0=\hbox{$\left[\box0\right]$}
  \advance\argwidth by -\wd0
 \left[\kern.3\argwidth\box0 \kern.3\argwidth\right]}

\newcommand{\cC}{\mathcal{C}}

\newcommand{\eE}{\mathcal{E}}
\newcommand{\fF}{\mathcal{F}}

\newcommand{\mM}{\mathcal{M}}

\newcommand{\oO}{\mathcal{O}}
\newcommand{\pP}{\mathcal{P}}

\newcommand{\sS}{\mathcal{S}}

\newcommand{\bA}{\mathbb{A}}
\newcommand{\bC}{\mathbb{C}}
\newcommand{\bL}{\mathbb{L}}
\newcommand{\bP}{\mathbb{P}}
\newcommand{\bT}{\mathbb{T}}

\newcommand{\fm}{\mathfrak{m}}
\newcommand{\fg}{\mathfrak{g}}
\newcommand{\fr}{\mathfrak{r}}
\newcommand{\fh}{\mathfrak{h}}
\newcommand{\fc}{\mathfrak{c}}

\renewcommand{\tilde}{\widetilde}

\newcommand{\tU}{\tilde{U}}

\newcommand{\Ob}{\mathcal{O}b}
\newcommand{\OB}{\mathop{\rm Ob}\nolimits}
\newcommand{\Hone}{\mathop{H^1}\nolimits}

\renewcommand{\hat}{\widehat}

\newcommand{\bom}{\bar{\omega}}
\newcommand{\coker}{\mathop{\rm coker}\nolimits}
\newcommand{\dspec}{\mathop{\mathbf{Spec}}\nolimits}
\newcommand{\lr}{\longrightarrow}
\newcommand{\Hom}{\mathop{\rm Hom}\nolimits}
\newcommand{\id}{\textrm{id}}
\newcommand{\Ext}{\mathop{\rm Ext}\nolimits}
\newcommand{\Spec}{\mathop{\rm Spec}\nolimits}
\newcommand{\Coh}{\mathop{\rm Coh}\nolimits}
\newcommand{\im}{\mathop{\rm im}\nolimits}
\newcommand{\Crit}{\mathop{\rm Crit}\nolimits}

\newcommand{\GL}{\mathop{\rm GL}\nolimits}

\def\Abul{A^\bullet}
\def\ck{\mathrm{ck}}
\def\fR{\mathfrak{R}}

\def\lal{_\lambda}
\def\lra{\longrightarrow}
\def\ti{\tilde}

\def\bl{\bigl(} \def\br{\bigr)}

\def\CC{\mathbb C}
\def\sta{^\ast}

\def\lalp{_\alpha}
\def\sub{\subset}
\def\cM{\mM}
\def\tcM{\tilde\cM}
\def\virt{^{\mathrm{vir}}}

\def\hW{\hat W}
\def\uss{^{ss}}
\def\beq{\begin{equation}}
\def\eeq{\end{equation}}

\def\git{/\!\!/}
\def\lalp{_\alpha }
\def\lab{_{\alpha\beta}}
\def\lbet{_\beta}

\def\rred{{\mathrm{red}}}
\def\hV{\hat V}
\def\hU{\hat U}
\def\uss{^{ss}}
\def\us{^s}

\makeatletter
 
  \@addtoreset{equation}{section}
\makeatother

\title{Generalized Donaldson-Thomas Invariants via Kirwan Blowups}

\author{Young-Hoon Kiem}
\address{Department of Mathematics and Research Institute of Mathematics, Seoul
National University, Seoul 08826, Korea}
\email{kiem@snu.ac.kr}

\author{Jun Li}
\address{Shanghai Center for Mathematical Sciences, Fudan University,  Shanghai, China}
\email{lijun2210@fudan.edu.cn}

\author{Michail Savvas}
\address{Department of Mathematics, The University of Texas at Austin, Austin, TX 78712, USA}
\email{msavvas@utexas.edu}

\begin{document}

\begin{abstract}
We develop a virtual cycle theoretic approach towards generalized Donaldson-Thomas theory of Calabi-Yau threefolds.
Let $\mM$ be the moduli stack of Gieseker semistable sheaves of fixed topological type on a Calabi-Yau threefold $W$. 

We construct an associated Deligne-Mumford stack $\tilde{\mM}$ with an induced semi-perfect obstruction theory of virtual dimension zero and define the generalized Donaldson-Thomas invariant of $W$ via Kirwan blowups
to be the degree of the virtual cycle $[\tilde{\mM}]^{\mathrm{vir}}$. We show that it is invariant under deformations of the complex structure of $W$. 
\end{abstract}

\maketitle

\section{Introduction}

\subsection{Background, history and results} Donaldson-Thomas (abbreviated as DT hereafter) invariants were first introduced by Thomas in \cite{Thomas} as enumerative invariants counting stable sheaves with fixed Chern character $\gamma$
on Calabi-Yau threefolds, when the coarse moduli space of such stable sheaves is proper. 
Following the virtual cycle construction of DT invariants, they are invariant under deformation of the Calabi-Yau threefold.

DT invariants constitute an important enumerative theory of Calabi-Yau threefolds and are conjecturally (and in many cases provably so, cf. \cite{MNOP2,MNOP1}) related to other enumerative invariants, such as Gromov-Witten invariants, Stable Pair \cite{PT1,PT2} and Gopakumar-Vafa invariants \cite{MaulikToda}.
\blfootnote{YHK was partially supported by Samsung Science and Technology Foundation SSTF-BA1601-01; JL was partially supported by NSF DMS-1564500 and DMS-1601211;
MS was partially supported by a Stanford Graduate Fellowship, an Alexander S. Onassis Foundation Graduate Scholarship and an A.G. Leventis Foundation Grant.}

Later, Joyce-Song constructed generalized
Donaldson-Thomas invariants (in short GDT invariants) for moduli of semistable sheaves of Chern character $\gamma$ (cf. \cite{JoyceSong}). 
Their construction relies on Behrend's motivic interpretation of DT invariants (cf. \cite{BehFun}), which makes working
with GDT relatively easy and their construction generalizable to moduli of complexes. However, their proof that the GDT invariants 
are invariant under deformation of the complex structure of the Calabi-Yau threefold is indirect, relying on wall crossing, and seems difficult to cover
more general cases.

The purpose of this work is to provide a new construction of GDT invariants which makes their deformation invariance 
transparent. 
Our approach is first making use of Kirwan's partial desingularization \cite{Kirwan} scheme
to transform the Artin moduli stack $\mM$ of Gieseker semistable sheaves of Chern character $\gamma$ into
a Deligne-Mumford (abbreviated as DM) stack $\tilde\mM$, while at the same time ``lift" certain 2-term perfect complexes on a cover of 
$\mM$ to a (semi-)perfect
obstruction theory of $\tilde\cM$, thus obtaining a virtual cycle $[\tcM]\virt$. We call 
the so defined generalized DT invariant the generalized DT invariant via Kirwan blowups (in short DTK invariant).

We state our main theorem. Let $W$ be a smooth projective Calabi-Yau threefold;
let $\gamma \in H^*(W, \mathbb{Q})$ be as before and let $\cM$ be the stack of Gieseker semistable 
sheaves on $W$ of Chern character $\gamma$, rigidified by the automorphisms generated by homothesis of 
every sheaf, so that the open substack of stable sheaves $\mM^s\subset\mM$ is a DM stack.\footnote{All stacks of (semi)stable sheaves in this paper are rigidified as such.}

\begin{thm} [Theorem-Definition~\ref{Gen DT}] Let $\gamma \in H^*(W, \mathbb{Q})$ and $\cM^s\sub \mM$ be as stated.
Then Kirwan's partial desingularization process constructs a proper DM stack $\tilde{\mM}$, called the intrinsic stabilizer reduction of $\mM$, and a morphism $\tilde{\mM}\to
\mM$ so that $\tilde{\mM}\to\mM$ is an isomorphism over
$\mM^s\subset\mM$, and $\tilde\mM$ is equipped with a semi-perfect obstruction theory of virtual dimension zero,
extending that of $\mM^s$. This semi-perfect obstruction theory induces a virtual cycle 
$[\tilde{\mM}]^{\mathrm{vir}}\in A_0(\tilde\mM)$, whose degree
$\deg [\tilde{\mM}]^{\mathrm{vir}}$ is invariant under deformation of the complex structure of $W$.
\end{thm}

We define 
$$\mathrm{DT}_\gamma(W) =\deg [\tilde{\mM}]^{\mathrm{vir}}
$$ 
to be the generalized DT invariant via Kirwan blowups, or DTK invariant.
\smallskip

It will be desirable to compare the GDT invariant via Kirwan blowups with the GDT invariant constructed by Joyce-Song.
We believe that the GDT invariant defined in this paper can be expressed as a ``universal" combination of the GDT invariants 
of Joyce-Song. The confirmation of such a relation will add to our understanding of GDT invariants of Calabi-Yau threefolds.

Besides the aforementioned work of Joyce-Song in \cite{JoyceSong}, there have been other works on constructing
(G)DT invariants, notably that of Kontsevich-Soibelman in \cite{KontSoibel}, and more recently by Behrend-Ronagh in \cite{BehRon2,BehRon1}. 

\subsection{Sketch of the construction} \label{construction outline}

Let $\mM$ be an Artin stack. We assume that $\mM$ is the truncation of a $(-1)$-shifted symplectic derived Artin stack, and $\mM$ is a GIT global quotient stack
$$\mM = [X / G],
$$
meaning that $G$ is reductive and $X$ is a $G$-invariant closed subscheme $X\sub P = (\bP^N)^{ss}$ so that $P$ is smooth and $G$ acts linearly on $\bP^N$ via a homomorphism $G\to GL_{N+1}$.
 
As $P$ is smooth, we can apply Kirwan's partial desingularization to get a $G$-equivariant morphism
$\pi:\tilde P\to P$ so that the GIT quotient $\ti P\git G$, which exists and is projective, is a geometric quotient with only possibly finite quotient singularities. In this paper, we call both $\ti P$ and the quotient $\ti P\git G$ the Kirwan partial desingularizations
of $P$ and $P\git G$, respectively. When the stable locus $P^s$ is empty, $\ti{P}^s$ is also empty, so $\ti P \git G = \emptyset$. We thus assume $P^s \neq \emptyset$ in what follows.

As $X\sub P$ is a closed subscheme, possibly singular, adapting the intrinsic blowup introduced by the first and second named authors \cite{KiemLi}, and applying the Kirwan partial desingularization procedure, we obtain a closed $G$-equivariant subscheme
$\ti X\sub\ti P$ fitting into the commutative squares
$$\begin{CD}
\tilde X @>{\sub}>> X\times_P\ti P @>{\sub}>> \ti P\\
@VVV @VVV@VVV\\
X@= X @>{\sub}>> P.
\end{CD}
$$

We define $\ti X$ and 
$$\ti \cM=[\ti X/ G]
$$
to be the intrinsic stabilizer reductions of $X$ and $\cM$, respectively.
We comment that $\ti X$ only depends on the $G$-structure on $X$, not the embedding $X\sub P$, and the inclusion
$\tilde X \sub X\times_P\ti P$ is strict when $X$ is singular along the locus where the stabilizers are positive dimensional. 
Thus $\ti\cM$ is canonically defined.

To construct the virtual cycle of $\ti\cM$, we use the theory of d-critical loci of Joyce, 
developed in \cite{JoyceDCrit}. Since $\mM$ is the truncation of a $(-1)$-shifted symplectic derived Artin stack, by \cite{JoyceArt} it is a d-critical Artin stack, which implies that
$X$ is a $G$-invariant d-critical locus. Then for every $x \in X$ such that $G \cdot x$ is closed in $X$ with $H$ denoting the stabilizer of $x$ in $G$ (hence reductive), we have $G$-invariant affine Zariski open $x \in U \sub X, x \in V \sub P$ and locally closed $H$-invariant affine subschemes $T \sub U,\ S \sub V$ granted by Luna's \'etale slice theorem such that there exists a diagram
\begin{equation}
\xymatrix{
[T / H] \ar[d]  \ar[r] & [ U / G] \ar[d] &\\
[S / H] \ar[r] & [V / G] 
}
\end{equation}
with \'{e}tale horizontal arrows and $T = ( df = 0 ) \subseteq S$ for $f \colon S \to \bA^1$ an $H$-invariant regular function on $S$.

Thus we have the following $H$-equivariant 4-term complex
\begin{equation}\label{e5}
\fh=\mathrm{Lie}(H)\lra T_{S}|_{T} \xrightarrow{d(df)^\vee} F_{S}|_{T}=\Omega_{S}|_{T} \lra \fh^\vee.
\end{equation}
For $x \in T$ with finite stabilizer, the first arrow is injective and the last arrow is surjective. Hence, \eqref{e5} is quasi-isomorphic to a 2-term complex which provides a perfect obstruction theory of $[T / H]$ and thus of $[U/G]$ near $x$.

In general, let $\ti x \in\ti X$ be lying over $x \in X$ with stabilizer $R$. Then we can lift \eqref{e5} canonically and find an
\'etale neighborhood $[ \ti T / R] \to [ \ti S / R ] \to [\ti P / G ]$ of the orbit of $\ti x$ in $[\ti P / G]$, a vector bundle $F_{\ti S}$ over $\ti S$ with an invariant section $\omega_{\ti S} \in H^0(\ti S, F_{\ti S})$ such that $\ti T = ( \omega_{\ti S} = 0 ) \sub \ti S$,
and a divisor $D_{\ti S}$, all $R$-equivariant, such that \eqref{e5} lifts canonically to a sequence
\begin{equation}\label{e6}
\fr = \mathrm{Lie}(R) \lra T_{\widetilde{S}}|_{\widetilde{T}} \lra F_{\widetilde{S}}|_{\widetilde{T}} \lra \fr^\vee(-D_{\widetilde{S}})
\end{equation}
whose first arrow is injective 
and last arrow is surjective. Therefore, \eqref{e6} is quasi-isomorphic to a 2-term complex
\begin{equation}\label{e7} 
( d \omega_{\ti S}^\vee )^\vee : T_{[\widetilde{S}/R]}|_{\widetilde{T}} \lra F^\rred_{\widetilde{S}}|_{\widetilde{T}},
\end{equation}
where $F^\rred_{\widetilde{S}}$ is the kernel of the last arrow in \eqref{e6}. 
Quotienting by $R$, we get
\begin{equation}\label{e7} 
( d (\omega_{\ti S}^\rred)^\vee )^\vee : T_{[\widetilde{S}/R]}|_{[\widetilde{T}/R]} \lra 
F_{[\widetilde{S}/R]}^\rred|_{[\widetilde{T}/R]}.
\end{equation}

We will show that the collection of the cokernels $\coker( d (\omega_{\ti S}^\rred)^\vee )^\vee$
patch to a coherent sheaf $\Ob_{\ti \cM}$ of
$\oO_{\ti\cM}$-modules, and that the symmetric obstruction theories of the various $T$ defined by $(df=0)$ induce a
semi-perfect obstruction theory on $\ti\cM = [\ti X/G]$, with obstruction sheaf $\Ob_{\ti \cM}$. (See \cite{LiChang} for the definition and properties of
semi-perfect obstruction theories.)

Using this semi-perfect obstruction theory in the case of Gieseker semistable sheaves of Chern character $\gamma$, we obtain a virtual cycle $[\ti\cM]\virt$ whose degree we define to be the generalized DT invariant via Kirwan blowups
\begin{equation}\label{e8} 
\mathrm{DT}_\gamma(W) = \deg [\widetilde{X}/G]^{\mathrm{vir}}.
\end{equation}

Finally, the relative version of the above construction can be carried out along parallel lines, using the machinery of derived symplectic geometry. This implies that 
$\mathrm{DT}_\gamma$ is invariant under deformation of the underlying Calabi-Yau threefold $W$.

\subsection{Further work and related results} In the present paper, we construct generalized DT invariants which act as counts of semistable sheaves by introducing the construction of the intrinsic stabilizer reduction $\tilde{\mM}$ of a d-critical global quotient stack $\mM$ and producing a semi-perfect obstruction theory on $\tilde{\mM}$. These results have been generalized in several directions. 

In \cite{Sav}, the third named author uses recent technical results on Artin stacks with good moduli spaces \cite{Alper, AHR2, AlpHalpHein} and stability conditions \cite{familystab} to extend the construction to the case of semistable perfect complexes. 

In \cite{KiemSavvas, KiemSavvasLoc}, the first and third named authors develop a general framework for handling $K$-theoretic invariants. Generalized DT invariants of sheaves and complexes via Kirwan blowups fall into this setting and thus admit a $K$-theoretic refinement.

We comment that the construction of the intrinsic stabilizer reduction presents independent interest from the point of view of derived algebraic geometry and could be related to an equivariant version of the derived blowup developed in \cite{Hekking}.

Finally, Edidin-Rydh \cite{EdidinRydh} have also developed a blowup procedure for Artin stacks with good moduli spaces which canonically eliminates stabilizers. Our construction differs from theirs and we comment on the relation between the two at the end of Section~\ref{Kirwan blow-up}. 

\subsection{Outline of the paper} In Section~\ref{Kirwan blow-up} we review Kirwan's partial desingularization for GIT quotients of smooth 
projective schemes and intrinsic blowups and their role in generalizing Kirwan's procedure to obtain the so called intrinsic stabilizer reduction of general projective GIT quotients. Section~\ref{dCrit} contains relevant material about d-critical loci, while Section~\ref{SemiPerf} addresses semi-perfect obstruction theories. In Section~\ref{LoCalc}, we gather several local computations and formalize definitions that are helpful in the rest of the paper. Finally, in Section~\ref{gendt} we use all of the above to construct generalized DT invariants. In Section~\ref{definv} we combine the absolute case and the machinery of derived algebraic geometry and $(-1)$-shifted symplectic structures to show their deformation invariance.

\subsection{Notation and conventions} Here are the various notations and other conventions that we use throughout the paper:
\begin{itemize}
\item[--] All varieties and schemes are defined over $\bC$.

\item[--] $C$ denotes a smooth quasi-projective curve over $\bC$. 

\item[--] $G$, $H$ denote complex reductive groups. Usually, $H$ will be a subgroup of $G$.

\item[--] If $U$ is a scheme over $C$, then $d_{U/C}$ denotes the relative de Rham differential for $U$. When $C$ is clear from context or suppressed in the notation, we often just write $d_U$.

\item[--] If $U \hookrightarrow V$ is an embedding, $I_{U \subseteq V}$ or $I_U$ (when $V$ is clear from context) denotes the ideal sheaf of $U$ in $V$.

\item[--] For a morphism $\rho \colon U \to V$ and a sheaf $E$ on $V$, we often use $E \vert_U$ to denote $\rho^* E$.

\item[--] If $U$ is a $G$-scheme, a point $u \in U$ is considered stable if its stabilizer is finite and its $G$-orbit is closed in $U$.

\item[--] If $V$ is a $G$-scheme, $V^G$ and $Z_G$ are both used to denote the fixed point locus of $G$ in $V$.

\item[--] If $U$ is a scheme with a $G$-action, then $\hU$ is used to denote the Kirwan blowup of $U$ with respect to $G$. $\tU$ denotes the intrinsic stabilizer reduction of $U$. Typically, whenever we perform a Kirwan blowup, $G$ will be connected, so we freely (and implicitly) assume this throughout.

\item[--] If $V$ is a $G$-scheme with an equivariant bundle $F_V$ and an invariant global section $\omega_V$, $F_{\hV}$ and $\omega_{\hV}$ denote the induced bundle and global section (see Section~\ref{Kirwan blow-up}) on the Kirwan blowup $\hV$.
Note that we follow the convention that the intrinsic stabilizer reduction is the result of a sequence of
Kirwan blowups.

\item[--] We often use the abbreviations DT, DM and GIT for Donaldson-Thomas, Deligne-Mumford and Geometric Invariant Theory respectively.

\end{itemize}

\section{Intrinsic Stabilizer Reduction} \label{Kirwan blow-up}

In this section, we work with the following situation, which is standard in Geometric Invariant Theory (GIT). Let $G$ be a reductive group; $P$ is a smooth scheme which is the semistable locus of a smooth projective variety with linearized $G$-action, and
$$X\sub P 
$$
is a $G$-invariant closed subscheme. Thus, we have a $G$-equivariant closed embedding $P \sub (\bP^N)\uss$, where by $(\bP^N)\uss$ (resp. $(\bP^N)\us$) we mean the open subscheme of GIT semistable (resp. stable) points under the linear $G$-action $G \to \GL_{N+1}$. (For GIT, see \cite{MFK}.) 

We start by briefly recalling Kirwan's desingularization procedure for smooth GIT quotients. This blowing up procedure, applied to the smooth scheme $P$, produces a new smooth scheme $\ti{P}$ with a linearized $G$-action satisfying $\ti{P}^s = \ti{P}^{ss}$ and a $G$-equivariant blowdown map $\ti{P} \to P$. When $P^s = \emptyset$, we have that $\ti{P}^s = \emptyset$ so we may assume that $P^s \neq \emptyset$ in what follows.

We then proceed to generalize the procedure to possibly singular GIT quotients. Using Luna's \'{e}tale slice theorem and adapting the intrinsic blowups introduced in \cite{KiemLi}, we may to define a $G$-invariant closed subscheme $\ti{X} \sub \ti{P}$ fitting in a commutative diagram
\begin{align*}
\xymatrix{
\ti{X} \ar[r] \ar[d] & \ti{P} \ar[d] \\
X \ar[r] & P.
}
\end{align*}

We show that $\ti{X}$ is independent of the choice of embedding $X \sub P$ and hence canonical and define $\ti{X}$, $[\ti{X} / G]$ and $\ti{X} \git G$ to be the intrinsic stabilizer reductions of $X$, $[X/G]$ and $X \git G$ respectively. Our main interest in the subsequent sections will be to use the Deligne-Mumford stack $[\ti{X} / G]$ in our construction of DT invariants counting semistable sheaves.

Finally, we note that our construction can be performed in greater generality than the case of global quotient stacks $[X / G]$ coming from GIT. The appropriate general setting is Artin stacks with good moduli spaces satisfying mild conditions. This is carried out in detail in the sequel \cite{Sav} to this paper, which extends the construction of DT invariants to semistable complexes, using recent technical results on stacks and stability conditions. Even though we are interested in the simpler GIT setting here, we frame our results in this language as well for the benefit of the reader.

\subsection{Kirwan's blowup algorithm for smooth schemes} \label{Kirwan algo}

We review Kirwan's partial desingularization $\tilde P$ of $P$. 

As the stabilizer groups of points in $P\us$ are all finite, $P^{s}\git G=P\us/G$
can have at worst finite quotient singularities. In stack language, this means that the quotient stack $[P^{s}/G]$ is a DM stack. 
When $P\us\ne P$, the GIT quotient $P \git G$ may have worse than finite quotient singularities, and the quotient stack is not DM. In \cite{Kirwan}, Kirwan produced a canonical procedure to blow up $P$ in order to produce a DM stack out of the Artin stack $[P/G]$. 

If the orbit of an $x\in P$ is closed in $P$, then the stabilizer $G_x$ of $x$ is a reductive subgroup of $G$. Let us fix a representative of each conjugacy class of subgroups $R$ of $G$ that appear as the identity component of the stabilizer $G_x$ of an $x\in P$ with $G\cdot x$ closed in $P$. Let $\fR(P)$ denote the set of such representatives. 
By \cite{Kirwan}, $\fR(P)$ is finite and $\fR(P)=\{1\}$ if and only if $P=P^s$. 

Let $R\in \fR(P)$ be an element of maximal dimension and let $Z_R$ be the fixed locus by the action of $R$, which is smooth. 
Then $GZ_R=G\times_{N^R}Z_R$ is smooth in $P$, where $N^R$ is the normalizer of $R$ in $G$.

We let $\pi: \mathrm{bl}_R(P) \to P$ be the blowup of $P$ along $GZ_R=G\times_{N^R}Z_R$.
Then $L=\pi^*\oO_P(1)(-\epsilon E)$ is ample for $\epsilon>0$ sufficiently small, 
where $E$ denotes the exceptional divisor of $\pi$. The action of $G$ on $P$ induces a linear action of $G$ on $\mathrm{bl}_R(P)$ with respect to $L$. 
The semistable points in the closure of $\mathrm{bl}_R(P)$ inside the projective space given by the embedding induced by the ample line bundle $L$ all lie in $\mathrm{bl}_R(P)$. We have the following important property, which in this setting was shown by Kirwan and in greater generality by Reichstein.

\begin{thm} \emph{\cite[Lemma~7.9, Remark~7.17]{Kirwan} \cite[Theorem~2.3, 2.4]{Reichstein}} \label{reichstein}
Let $Z = GZ_R \sub P$ as above and denote $q \colon P \to P \git G$. The unstable locus of $\mathrm{bl}_R(P)$ is the strict transform of the saturation $q^{-1}( q(Z) )$ of $Z$.
\end{thm}

It follows that the unstable points in $\mathrm{bl}_R(P)$ are precisely those points whose orbit closure meets the unstable points in $E=\bP N_{GZ_R/P}$, and the unstable points of $\mathrm{bl}_R(P)$ lying in the fiber 
$E|_x=\bP N_{GZ_R/P}|_x$ for $x\in GZ_R$ are precisely the unstable points of the projective space 
$\bP N_{GZ_R/P}|_x$ with respect to the linear action of $R$ by \cite[Lemma~7.8]{Kirwan} . 
We define 
$$\hat P=(\mathrm{bl}_R(P))\uss.
$$
By \cite{Kirwan}, $\fR(\hat P)=\fR(P)-\{R\}$.

\begin{defi} The scheme
$\hat P$ (resp. $\hat P \git G$) is called the Kirwan blowup of $P$ (resp. $P \git G$) with respect to the group $R$.
\end{defi}
Therefore, repeating the Kirwan blowup finitely many times, once for each element of $\fR(P)$ in order of decreasing
dimension, we end up with a $G$-equivariant morphism
$$\widetilde{P}\lra P$$
which induces a projective morphism
$$\widetilde{P} / G\lra P\git G.$$
As $\widetilde{P}$ is smooth,
$\widetilde{P}/G$ has at worst finite quotient singularities.

\begin{defi}
The scheme $\widetilde{P}$ (resp. $\tilde{P} / G$) is called the Kirwan partial desingularization of $P$
(resp. $P\git G$).
\end{defi}

\begin{rmk} \label{local GIT}
We note here that in the Kirwan blowup,  
we can detect which points on the exceptional divisors that occur are unstable just by looking at the action of $R$ on $\bP N_{GZ_R/P}$. Furthermore, a point off the exceptional divisor is unstable if the closure of its orbit meets the unstable locus of the exceptional divisor. 
Thus for any smooth affine $G$-scheme $V$, we can define its Kirwan blowup $\hat V$ associated to any $R\in\fR(V)$
of maximal dimension in the same way.
\end{rmk}

\subsection{Intrinsic blowups} 

Suppose that $U$ is a scheme with an action of a reductive group $G$. Suppose moreover that we have an equivariant embedding $U \to V$ into a smooth $G$-scheme $V$ and let $I$ be the ideal defining $U$. Since $U \subset V$ is $G$-equivariant, $G$ acts on $I$ and we have a decomposition $I = I^{fix} \oplus I^{mv}$ into the fixed part of $I$ and its complement as $G$-representations. 

Let $V^G$ be the fixed point locus of $G$ inside $V$ and $\pi \colon \mathrm{bl}_G(V) \to V$ the blowup of $V$ along $V^G$. 
Let $E\sub\mathrm{bl}_G(V)$ be its exceptional divisor and $\xi \in \Gamma( \oO_{\mathrm{bl}_G(V)}(E))$ the tautological defining equation of $E$. 
We claim that
\begin{equation}\label{xi}
\pi^{-1} (I^{mv}) \subset \xi \cdot \oO_{\mathrm{bl}_G(V)}(-E) \subset \oO_{\mathrm{bl}_G(V)}.
\end{equation}
Let $\mathcal R: I\to I^{fix}$ be the Reynolds operator. Then for any $\zeta\in I^{mv}$, $\mathcal R(\zeta)=0$. Let 
$x\in V^G$ be any closed point. Since $\oO_x$ is fixed by $G$, we have
$\zeta|_x=\mathcal R(\zeta)|_x=0$. This proves that all elements in $I^{mv}$ vanish along $V^G$, hence \eqref{xi}.

Consequently, $\xi^{-1} \pi^{-1} (I^{mv}) \subset \oO_{\mathrm{bl}_G(V)}$. 
We define $I^{intr}\sub\oO_{\mathrm{bl}_G(V)}$ to be 
\begin{align} \label{tilde I}
I^{intr} = \text{ideal generated by } \pi^{-1}(I^{fix}) \text{ and } \xi^{-1} \pi^{-1} (I^{mv}).
\end{align}
\begin{defi} \emph{(Intrinsic blowup)} \label{intrinsic blow-up}
The $G$-intrinsic blowup of $U$ is the  subscheme $U^{intr}\sub \mathrm{bl}_G(V)$ defined by the ideal $I^{intr}$. 
\end{defi}

\begin{lem} \label{ideal of intrinsic}
The $G$-intrinsic blowup of $U$ is independent of the choice of $G$-equivariant embedding $U \subset V$,
and hence is canonical.
\end{lem} 

\begin{sproof} 
The proof is identical to \cite[Section~3.1]{KiemLi}. We give a very brief account of the main steps involved.

One firstly establishes the claim in the case of $U$ being a formal affine scheme such that the fixed locus $U^G$ is an affine scheme which has the same support as $U$ does. 

For an affine scheme $U$ with a $G$-action, one can then take the formal completion ${U}^c$ of $U$ along 
$U^G$ and check that the $G$-intrinsic blowup of ${U}^c$ glues naturally with $U - U^G$ to yield the $G$-intrinsic blowup of $U$.

Finally, one can show by taking a cover by (Zariski or \'{e}tale) open affine schemes that the intrinsic blowup is well-defined for a general scheme or DM stack with a (representable) action of $G$.
\end{sproof}

\begin{rmk} \label{passage to formal}
The following remarks on the proof of Lemma~\ref{ideal of intrinsic} are in order:
\begin{enumerate}
\item If $U$ is smooth, then the $G$-intrinsic blowup coincides with the blowup of $U$ along $U^G$.
\item Since the core of the proof relies on working first in the formal completion of $V^G$ inside $V$ and proving the Lemma~\ref{ideal of intrinsic} in the case of formal schemes, this enables us to perform local calculations formally (or analytically) locally, if desired for convenience.
\end{enumerate}
\end{rmk}

Suppose $U$ is an affine $G$-scheme, then we can think of all points of $U$ as semistable. We can make sense of semistable points in $U^{intr}$ without ambiguity by considering $G$-equivariant embeddings $U \hookrightarrow V$ and using Remark~\ref{local GIT}. It is easy to check that the semistable locus is independent of the choice of particular embedding.

\begin{defi}[Kirwan blowup]
We define the Kirwan blowup of a possibly singular affine $G$-scheme $U$ associated with $G$ to be
$\hat U=(U^{intr})\uss$.
\end{defi}

\begin{exam} \label{example of Kirwan blowup}
Suppose that $G = \bC^\ast$ is the one-dimensional torus acting on the affine plane $V = \bC^2_{x,y}$ with weights $1$ and $-1$ on the coordinates $x$ and $y$ respectively and $U \sub V$ is the closed $G$-invariant subscheme cut out by the ideal $I = (x^2 y, x y^2)$. 

$V^{intr}$ is the blowup of $V$ along the fixed locus $V^G = \lbrace 0 \rbrace$. The unstable points are the punctured $x$-axis and $y$-axis together with the points $0, \infty$ of the exceptional divisor $\bP^1$. Thus we have that $\hV = \Spec [u ,v, v^{-1} ]$ where $G$ acts on $u, v$ with weights $1, -2$ respectively and the blowdown map $\hV \to V$ is given on coordinates by $x \mapsto u, y \mapsto uv$.

$\hU \sub \hV$ is the closed subscheme cut out by the ideal $(u^2)$, so that 
$$\hU = \Spec \left( \bC[u,v,v^{-1}] / (u^2) \right).$$

\end{exam}

\subsection{Blowup bundle and section}
Let $V$ be a smooth affine $G$-scheme, $F_V$ a $G$-equivariant vector bundle on $V$ and $\omega_V\in\Gamma(V,F_V)^G$
a $G$-invariant section. Then $U=(\omega_V=0)$ is a $G$-invariant subscheme of $V$.
Since $G$ is reductive, we have a decomposition into fixed and moving components
\begin{align}
F_V |_{V^G} = F_V|_{V^G}^{fix} \oplus F_V|_{V^G}^{mv}.
\end{align}

\begin{defi} \emph{(Blowup bundle)} \label{blow-up bundle}
Let $\pi : \hat{V} \rightarrow V$ be the Kirwan blowup of $V$ associated with $G$.
The blowup bundle of $F_V$, denoted by $F_{\hV}$,
is 
$$
F_{\hV} := \ker\bl \pi^* F_V \lra \pi^* (F_V|_{V^G}) \lra
 \pi^* (F_V |_{V^G}^{mv}) \br.
$$ 
The blowup section
$$\omega_{\hV} \in \Gamma(\hat V, F_{\hV}),
$$
is the lift of $\omega_V$, which exists since $\pi^* \omega_V$ maps to zero in $\pi^*( F_V|_{V^G}^{mv})$.
\end{defi}

To see that $F_{\hV}$ is indeed a vector bundle, we may argue as follows: The question is local, so we may assume that $F_V$ is trivial. Since $G$ is reductive, we can embed $V$ $G$-equivariantly into a $G$-representation $W = \bA^n$ with a trivial bundle $F_W = \bA^n \times \bA^m$ which is also a $G$-representation for the diagonal $G$-action and a split surjection $F_W^\vee |_V \to F_V^\vee$, giving rise to a split injection $F_V \to F_W|_V$ with locally free cokernel, so that $F_V$ is a subbundle of $F_W|_V$. As the embedding $V \sub W$ is regular, we have a Cartesian diagram of Kirwan blowups
\begin{align*}
\xymatrix{
\hV \ar[r] \ar[d] & \hat{W} \ar[d] \\
V \ar[r] & W
}
\end{align*}
and it is easy to check that $F_{\hV}$ is still a direct summand of $F_{\hW}|_{\hV}$.

We are thus reduced to the case of a trivial $G$-equivariant vector bundle $F_W = W \times \bA^m$ on a $G$-representation $W = \bA^n$, where $G$ acts diagonally and linearly on the two factors $W$ and $\bA^m$. Hence we may write $F_W = F_W^{fix} \oplus F_W^{mv}$ and it is immediate that $F_{\hW} = \pi^\ast F_W^{fix} \oplus \pi^\ast F_W^{mv}(-E)$, where $E$ is the exceptional divisor of the blowup $\pi \colon \hW \to W$, giving that $F_{\hW}$ is a vector bundle. As $F_{\hV}$ is a direct summand of $F_{\hW}|_{\hV}$, it is a vector bundle.

\begin{prop} \label{coinc}
Let $U' \subset \hat{V}$ be defined by the vanishing of $\omega_{\hV}$.
Then $U'$ is the Kirwan blowup $\hU$ of $U$. 
\end{prop}

\begin{proof} Let $I$ be the ideal of $U$ in $V$, generated by the section $\omega_V$. We need to check that the ideal $I^{intr}$ given by \eqref{tilde I} coincides with the ideal generated by $\omega_{\hV}$. By the above, it suffices to work locally. But in local coordinates $\omega_{\hV}$ is obtained from $\pi^* \omega_V$ by multiplying the moving components with $\xi^{-1}$, which immediately implies the claim.
\end{proof}

\subsection{Intrinsic stabilizer reduction for singular schemes} \label{intr part desing} 

We continue working with the $G$-equivariant embedding $X\sub P$ mentioned at the beginning of this section.
We list 
$$\fR(P)=\{R_1, \dots, R_m,\{1\}\}$$ 
in order of decreasing dimension.

We begin with $R=R_1 \in \fR(P)$. 
For any $x \in Z_{R} \sub P$, let $S$ be an \'{e}tale affine slice for $x$ in $P$, provided by Luna's \'{e}tale slice theorem (cf. \cite[Theorem~5.3]{Drezet}), and let 
$T = S\times_P X$. (In case $x\not\in X$, we can choose $S$ so that $T=\emptyset$.)

As $S_{}$ is smooth, affine and $R$-invariant, we let $\hat S$ be the Kirwan blowup of $S$ associated with $R$.
As $T\sub S$ is closed and $R$-invariant, we let $\hat T\sub\hat S$ be the Kirwan blowup of $T$ associated with $R$.
They fit into a commutative diagram
\begin{align}
\xymatrix{
G \times_R \hat{T}_{} \ar[r] \ar[d] & G \times_R \hat{S}_{} \ar[r] \ar[d] & \hat{P} \ar[d] \\
G \times_R {T}_{} \ar[r] \ar[d] & G \times_R {S}_{} \ar[r] & {P} \\
X \ar[rru]  
}
\end{align}
This collection of \'{e}tale maps $G \times_R S_{}\to P$ cover the locus $GZ_R$ inside $P$.
Let $E\sub \hat P$ be the exceptional divisor of $\hat P\to P$.
Because $P-GZ_R =\hat P-E$, $P-GZ_R$ can be viewed as an open subscheme of $\hat P$.
Consequently, the collection of \'etale maps $G \times_R \hat S_{}\to \hat P$ together with $P-E$ form an \'etale covering of $\hat P$.

We next consider the collection of all possible $G\times_R\hat T\sub G\times_R\hat S$.

\begin{prop} \label{indep} The collection of $G\times_R\hat T\to \hat P$ just mentioned, together with $X-GZ_R\sub \hat P-E$, form a closed subscheme $\hat X\sub\hat P$, called the Kirwan blowup of $X$.
Further, $\hat{X}$ is canonical in that it is independent of the choice of slices or choice of projective embedding.
\end{prop}

\begin{proof}
We first show the independence from the particular choice of slices.

Let $S_1, S_2$ be two \'{e}tale slices in $P$, such that $T_1 = S_1 \cap X, \ T_2 = S_2 \cap X$ are the induced slices for $X$. Let $I_1, I_2$ be the ideal sheaves of $T_1 \sub S_1$ and $T_2 \sub S_2$ respectively.

Near every point in $P$ covered by $S_1$ and $S_2$, we can find a common \'{e}tale refinement $S_{12}$. This can be seen as follows: Since $[P/G]$ has affine diagonal, the fiber product $S_1 \times_{[P/G]} S_2$ is an affine scheme with a $(R \times R)$-action. For any point $z \in P$ fixed by $R$ we may take a slice $S_{12}$ for $z$ in $S_1 \times_{[P/G]} S_2$.

Consider the composition $p_i \colon S_{12} \to S_1 \times_{[P/G]} S_2 \to S_i$. Since $S_{12}$ and $S_i$ are smooth of the same dimension and $p_i$ induces an isomorphism on tangent spaces at $z$, it must be \'{e}tale (up to shrinking). It is also evidently $R$-equivariant and hence is indeed an \'{e}tale refinement of $S_1$ and $S_2$. It follows that $p_1^* I_1 \simeq p_2^* I_2 \simeq I_{12}$ as ideal sheaves on $S_{12}$, defining $T_{12} \sub S_{12}$.

Taking Kirwan blowups commutes with \'{e}tale base change. We thus obtain induced \'{e}tale maps $\hat{p}_i \colon \hat{S}_{12} \to \hat{S}_i$, such that $\hat{p}_1^\ast I_1^{intr} \simeq \hat{p}_2^\ast I_2^{intr} \simeq I_{12}^{intr}$, defining the subscheme $\hat{T}_{12} \sub \hat{S}_{12}$.

Since we may cover $(G \times_R S_1) \times_P (G \times_R S_2)$ by \'{e}tale opens of the form $G \times_R S_{12}$ around the fixed points of $R$, \'{e}tale descent implies that we obtain a well-defined closed subscheme $\hat X \sub \hat P$.

Regarding the choice of projective embedding, we may cover $X$ by $G$-invariant affine opens $U\lalp \sub X$. Let $U\lalp \to V\lalp$ be equivariant embeddings into smooth $G$-schemes. Using those, we may define $\hat U\lalp \sub \hat V\lalp$ by first taking intrinsic blowups and restricting to semistable points. We observe that by Remark~\ref{local GIT} the latter restriction is unambiguous, as the unstable points of the intrinsic blowup are the ones whose $G$-orbit closure intersects the unstable locus of the exceptional divisor. Therefore, by the canonical nature of intrinsic blowups, for each $\alpha$ the Kirwan blowup $\hat U\lalp$ is independent of the local embedding $U\lalp \to V\lalp$. Since we may choose those to come from a $G$-invariant open cover $V\lalp \sub P$ of any projective embedding $X \sub P$, it follows that $\hat X$ is independent of the choice of projective embedding. \end{proof}

We let $P_1=\hat P$ and $X_1=\hat X$ be their respective Kirwan blowups associated with $R_1$. Then
$X_1\sub P_1$, and $\fR(P_1)=\{R_2,\cdots, R_m,\{1\}\}$.
We let $X_2\sub P_2$ be $\hat {X_1}\sub \hat{P_1}$, the Kirwan blowups associated with $R_2$, and so on, until we obtain
$X_m\sub P_m$, having the property $\fR(P_m)=\{1\}$.

We denote 
$$\ti X=X_m,\quad \ti P=P_m.
$$

\begin{defi}\label{def-3.8}
We call $\ti X$ and $\tilde{\mM} = [ \tilde{X} / G ]$ the intrinsic stabilizer reduction of $X$ and the Artin stack $\mM = [ X / G ]$, respectively.
\end{defi}

\begin{rmk}\label{dense}
When $X^s$ is dense in $X$, then $\hat X\to X$ is birational and $\tilde{\mM} \to \mM$ is proper and birational. The other extreme case is when $X^G=X$, then
$\hat X=\emptyset$. In general there are cases when $(X^G)_{\text{red}}=X_{\text{red}}$ and $\hat X\ne \emptyset$. One such case is given in Example~\ref{example of Kirwan blowup}.
\end{rmk}

\subsection{Intrinsic stabilizer reduction for stacks with good moduli spaces} So far in this section we have defined the intrinsic stabilizer reduction $\tilde{\mM}$ for a global quotient stack $\mM = [X / G]$ coming from GIT. 

In our construction, we have used Luna's \'{e}tale slice theorem which allows us to work locally on Cartesian diagrams of the form
\begin{align} \label{etale slice diagram}
\xymatrix{
[U / G_x ] \ar[d] \ar[r] & \mM \ar[d] \\
U \git G_x \ar[r] & M,
}
\end{align}
where $x$ ranges through all the closed points of $\mM$ with (reductive) stabilizer $G_x$ and $M = X \git G$. These diagrams provide an \'{e}tale cover for the stack $\mM$.

More generally, the construction can be applied to Artin stacks $\mM$ with good moduli space $\mM \to M$, introduced in \cite{GoodAlper}, which are a generalization of global quotient stacks $\mM = [X / G]$ with a morphism $\mM \to M = X \git G$ coming from GIT. The \'{e}tale slice theorem proved in \cite{Alper} then implies the existence of the diagrams \eqref{etale slice diagram}.

This is carried out carefully in full generality in the sequel \cite{Sav}, which extends the techniques of the present paper to define generalized DT invariants of moduli stacks of semistable complexes. These admit good moduli spaces but do not come from GIT. 

In this context, Edidin-Rydh have also developed a blowup procedure for stacks $\mM$ with good moduli spaces in \cite{EdidinRydh}. For smooth quotient stacks $\mM = [X / G]$, our construction produces the same result as theirs, which is the original Kirwan partial desingularization of $\mM$. For singular stacks, our Kirwan blowups can be phrased in their language of saturated blowups, however the stack they obtain is generally a closed substack of our intrinsic stabilizer reduction $\tilde{\mM}$, as the following example showcases. 
We would like to thank them for kindly informing us of their work.

\begin{exam}
Similarly to Example~\ref{example of Kirwan blowup}, let $G= \bC^\ast$ act on $V = \bC^2_{x,y}$ with weights $1, -1$ on the coordinates $x, y$ respectively and define $U \sub V$ to be the $G$-invariant closed subscheme cut out by the ideal $I = (xy)$.  

The Kirwan blowup of the scheme $U$ is $\hU = \Spec \left( \bC[u,v,v^{-1}] / (u^2) \right)$.

On the other hand, the fixed locus of $U$ is $U^G = \lbrace 0 \rbrace$ with the reduced scheme structure. Thus, $U \git G = \Spec \left( \bC [xy] / (xy) \right)$ and $U^G \git G$ define the same scheme, and the saturated blowup of $U$ with center $U^G$ defined in \cite{EdidinRydh} will be empty. This is clearly different from the Kirwan blowup $\hU$, which is non-empty.
\end{exam}

\section{d-critical Loci} \label{dCrit}

In this section, we recall Joyce's theory of d-critical loci, as developed in \cite{JoyceDCrit}, and establish some notation. 

We comment that there is also a parallel theory of critical virtual manifolds, developed in \cite{KiemLiCat}, which is equivalent to the theory of d-critical loci for the cases considered in this paper and could alternatively be used as well.

\subsection{d-critical schemes} \label{d-crit section} We begin by defining the notion of d-critical charts. Let $M$ be a scheme.

\begin{defi}\emph{(d-critical chart)} \label{d-crit chart}
A d-critical chart for $M$ is the data of $(U,V,f,i)$ such that: $U \subseteq M$ is Zariski open, $V$ is a smooth scheme, $f \colon V \to \bA^1$ is a regular function on $V$ and $U \xrightarrow{i} V$ is an embedding so that $U = (d_{V} f=0) = \Crit(f) \subseteq V$.

If $x \in U$, then we say that the d-critical chart $(U,V,f,i)$ is centered at $x$.
\end{defi}

Joyce defines a canonical sheaf $\sS_M$ of $\bC$-vector spaces with the property that for any Zariski open $U \subseteq M$ and an embedding $U \hookrightarrow V$ into a smooth scheme $V$ with ideal $I$, $\sS_M$ fits into an exact sequence
\begin{align}
0 \lr \sS_M |_U \lr \oO_V / I^2 \stackrel{d_{V}}{\lr} \Omega_{V}^1 / I \cdot \Omega_{V}^1.
\end{align}
\begin{exam} For a d-critical chart $(U,V,f,i)$ of $M$, the element $f+I^2 \in \Gamma(V,\oO_V/I^2)$ gives a section of $\sS_M|_U$.
\end{exam}

\begin{defi}\emph{(d-critical scheme)} \label{d-crit scheme}
A d-critical structure on a scheme $M$ is a section $s \in \Gamma(M, \sS_M)$ such that $M$ admits a cover by d-critical charts $(U,V,f,i)$ and $s|_U$ is given by $f+I^2$ as above on each such chart. We refer to the pair $(M,s)$ as a d-critical scheme.
\end{defi}

For a d-critical scheme $M$ and a Zariski open $U \subseteq M$, any embedding $U \hookrightarrow V$ into a smooth scheme can be locally made into a d-critical chart.

\begin{prop} \emph{\cite[Proposition~2.7]{JoyceDCrit}}\label{modify embedd}
Let $M$ be a d-critical scheme, $U \subseteq M$ Zariski open and $i \colon U \hookrightarrow V$ a closed embedding into a smooth scheme $V$. Then for any $x \in U$, there exist Zariski open $x \in U' \subseteq U$, $i(U') \subseteq V' \subseteq V$ and a regular function $f' \colon V' \to \bA^1$ such that $(U', V', f', i|_{U'})$ is a d-critical chart centered at $x$.
\end{prop}
In order to compare different d-critical charts, we need the notion of an embedding.

\begin{defi}
Let $(U,V,f,i)$ and $(R,W,g,j)$ be two d-critical charts for a d-critical scheme $(M,s)$ with $U \subseteq R$ Zariski open. We call a locally closed embedding $\Phi : V \to W$ an embedding between the two charts if $f = g \circ \Phi \colon V \to \bA^1$ and the following diagram commutes
\begin{align*} \notag
\xymatrix{
U \ar[r]^-{i} \ar[d] & V \ar[d]^{\Phi}\\
R \ar[r]^{j} & W.
}
\end{align*}
By abuse of notation, we use $\Phi : (U,V,f,i) \to (R,W,g,j)$ to denote this data.
\end{defi}

We then have the following way to compare different overlapping d-critical charts.

\begin{prop} \emph{\cite[Theorem~2.20]{JoyceDCrit}}\label{d-crit chart comparison}
Let $(U,V,f,i)$ and $(S,W,g,j)$ be two d-critical charts centered at $x$ for a d-critical scheme $(M,s)$. Then, after possibly (Zariski) shrinking $V$ and $W$ around $x$, there exists a d-critical chart $(T,Z,h,k)$ centered at $x$ and embeddings $\Phi : (U,V,f,i) \to (T,Z,h,k)$, $\Psi : (R,W,g,j) \to (T,Z,h,k)$.
\end{prop}

\subsection{Equivariant d-critical loci} For our purposes, we need equivariant analogues of the results of Subection~\ref{d-crit section}. The theory works in parallel as before (cf. \cite[Section~2.6]{JoyceDCrit}).

\begin{defi} \emph{(Good action)} Let $G$ be an algebraic group acting on a scheme $M$. We say that the action is good if $M$ has a cover $\lbrace U_\alpha \rbrace_{\alpha \in A}$ where every $U_\alpha \subseteq M$ is an invariant open affine subscheme of $M$.
\end{defi}

\begin{rmk}
If $M$ is obtained by GIT so that it is the semistable locus of a projective scheme with a linearized $G$-action, then the action of $G$ on $M$ is good.
\end{rmk}

It is straightforward to extend Definitions 2.1, 2.2 and 2.3 and Proposition~\ref{modify embedd} in the equivariant setting (cf. \cite[Definition~2.40]{JoyceDCrit}).

\begin{prop} \emph{\cite[Remark~2.47]{JoyceDCrit}} \label{equivariant d-crit prop}
Let $G$ be a complex reductive group with a good action on a scheme $M$. Suppose that $(M,s)$ is an invariant d-critical scheme. Then the following hold:
\begin{enumerate}
\item For any $x \in M$ fixed by $G$, there exists an invariant d-critical chart $(U,V,f,i)$ centered at $x$, i.e. an invariant open affine $U \ni x$, a smooth scheme $V$ with a $G$-action, an invariant regular function $f \colon V \to \bA^1$ and an equivariant embedding $i \colon U \to V$ so that $U = \Crit(f) \subseteq V$.
\item Let $(U,V,f,i)$ and $(S,W,g,j)$ be two invariant d-critical charts centered at the fixed point $x \in M$. Then, after possibly shrinking $V$ and $W$ around $x$, there exists an invariant d-critical chart $(T,Z,h,k)$ centered at $x$ and equivariant embeddings $\Phi : (U,V,f,i) \to (T,Z,h,k)$, $\Psi : (S,W,g,j) \to (T,Z,h,k)$.
\end{enumerate}
\end{prop}

\begin{rmk} \label{Rmk d-crit}
If $G$ is a torus $\left( \bC^{\ast} \right)^k$, then Proposition \ref{equivariant d-crit prop} is true without the assumption that $x$ is a fixed point of $G$.
\end{rmk}

\begin{rmk}\label{4.11}
There is a notion of d-critical locus for Artin stacks $\mM$ (cf. \cite[Section~2.8]{JoyceDCrit}). Then (cf. \cite[Example~2.55]{JoyceDCrit}) d-critical structures on quotient stacks $\left[ M / G \right]$ are in bijective correspondence with invariant d-critical structures on $M$.

Moreover, one may pull back d-critical structures along smooth morphisms between stacks.
\end{rmk}

\section{Semi-perfect Obstruction Theory} \label{SemiPerf}

This section contains necessary material about semi-perfect obstruction theories, as developed in \cite{LiChang}.

Let $U \to C$ be a morphism, where $U$ is a scheme of finite type and $C$ a smooth quasi-projective curve. We first recall the definition of 
perfect obstruction theory \cite{BehFan, LiTian}.

\begin{defi} \emph{(Perfect obstruction theory \cite{BehFan})} \label{Perf obs th}
A (truncated) perfect (relative) obstruction theory consists of a morphism $\phi \colon E \to L_{U/C}^{\geq -1}$ in $D^b(\Coh U)$ such that
\begin{enumerate}
\item $E$ is of perfect amplitude, contained in $[-1,0]$.
\item $h^0(\phi)$ is an isomorphism and $h^{-1}(\phi)$ is surjective.
\end{enumerate}
We refer to $\Ob_\phi := \Hone(E^\vee)$ as the obstruction sheaf of $\phi$.
\end{defi}

\begin{defi} \emph{(Infinitesimal lifting problem)} \label{Inf lift prob}
Let $\iota \colon \Delta \to \bar{\Delta}$ be an embedding with $\bar{\Delta}$ local Artinian, such that $I \cdot \fm = 0$ where $I$ is the ideal of $\Delta$ and $\fm$ the closed point of $\bar{\Delta}$. We call $(\Delta, \bar{\Delta}, \iota, \fm)$ a small extension. Given a commutative square
\begin{align}
\xymatrix{
\Delta \ar[r]^g \ar[d]^\iota & U \ar[d]\\
\bar{\Delta} \ar[r] \ar@{-->}[ur]_{\bar{g}} & C
}
\end{align}
such that the image of $g$ contains a point $p \in U$, the problem of finding $\bar{g} \colon \bar{\Delta} \to U$ making the diagram commutative is the ``infinitesimal lifting problem of $U/C$ at $p$".
\end{defi}

\begin{defi} \emph{(Obstruction space)} \label{Obs spaces}
For a point $p \in U$, the intrinsic obstruction space to deforming $p$ is $T_{p, U/ C}^1 := \Hone \left( (L_{U/C}^{\geq -1})^\vee \vert_p \right)$. The obstruction space with respect to a perfect obstruction theory $\phi$ is $\OB(\phi,p) := \Hone( E^\vee \vert_p )$.
\end{defi}

Given an infinitesimal lifting problem of $U/C$ at a point $p$, there exists by the standard theory of the cotangent complex a canonical element 
\begin{align}
\omega \left( g, \Delta, \bar{\Delta} \right) \in \Ext^1 \left( g^{\ast} L_{U/C}^{\geq -1} \vert_p, I\right) = T^1_{p, U/C} \otimes_\bC I
\end{align} 
whose vanishing is necessary and sufficient for the lift $\bar{g}$ to exist. 

\begin{defi} \emph{(Obstruction assignment)} \label{Obs assignment}
For an infinitesimal lifting problem of $U / C$ at $p$ and a perfect obstruction theory $\phi$ the obstruction assignment at $p$ is the element
\begin{align}
ob_U(\phi,g,\Delta,\bar{\Delta}) = h^1(\phi^\vee) \left( \omega \left( g, \Delta, \bar{\Delta} \right) \right) \in \OB(\phi,p) \otimes_\bC  I.
\end{align}
\end{defi}

Suppose now that $U$ is given by the vanishing of a global section $s \in \Gamma \left( V, F \right)$ where $F$ is a vector bundle on a scheme $V$ which is smooth over $C$. Let $J$ denote the ideal sheaf of $U$ in $V$ and $j \colon U \to V$ the embedding. Then we have a perfect obstruction theory given by the diagram
\begin{align}
\xymatrix{
E \ar@{=}[r] \ar[d]^-\phi & [ F^\vee \vert_U \ar[r]^-{d_{V/C} s^\vee} \ar[d]^-{s^\vee} & \Omega_{V/C} \vert_U \ar@{=}[d] ]\\
L_{U/\mM}^{\geq -1} \ar@{=}[r] & [ J/J^2 \ar[r]^-{d_{V/C}} & \Omega_{V/C} \vert_U ].
}
\end{align}

Since $V$ is smooth over $C$ we can find a lift $g' \colon \bar{\Delta} \to V$ of the composition $j \circ g$. Composing with the section $s \colon V \to F$ we obtain a morphism $s \circ g'  \colon \bar{\Delta} \to (g')^\ast F$. Since $g = g' \vert_\Delta$ factors through $U$, we must have $s \circ g' \in I \otimes_\bC F \vert_p$.

Let $\rho \colon I \otimes_\bC F \vert_p \to I \otimes_\bC \Ob_\phi \vert_p = I \otimes_\bC \OB(\phi,p)$ be the natural projection map.

\begin{lem} \emph{\cite[Lemma~1.28]{KiemLiCat}} \label{compute obs assign}
$ob_U(\phi,g,\Delta,\bar{\Delta}) = \rho \left( s \circ g'\right)$.
\end{lem}

\begin{proof}
$ob_U(\phi,g,\Delta,\bar{\Delta})$ is given by the composition
\begin{align}
g^\ast E \lr g^\ast L_{U/ C}^{\geq -1} \lr L_{\Delta / \bar{\Delta}}^{\geq -1} \lr I[1].
\end{align}
This fits into a commutative diagram
\begin{align}
\xymatrix{
g^* \Omega_{V / C} \ar@{=}[r] \ar[d] & g^* \Omega_{V / C} \ar[d] \ar[dr]^-{ob_V}\\
g^* E \ar[r]^-\phi \ar[d] & g^* L_{U / C}^{\geq -1} \ar[r] & I[1] \\
g^* {F^\vee[1]} \ar@<-0.6ex>@{-->}[urr]
}
\end{align}
Since $V$ is smooth over $C$, the map $ob_V$ must be zero. Using the distinguished triangle of the first column, we get a long exact sequence in cohomology
\begin{gather} \label{long exact}
\Hom(g^* \Omega_{V / C}, I) = I \otimes_\bC T_{V/C} \vert_p \xrightarrow{\left( d_{V/C} s^\vee \right)^\vee} \Hom(g^* F^\vee, I) = I \otimes_\bC F \vert_p\\ \notag
\lr \Ext^1(g^* E, I) \lr \Ext^1(g^* \Omega_{V / C}, I).
\end{gather}
Now, the fact that $ob_V$ is zero implies that $ob_U(\phi,g,\Delta,\bar{\Delta})$ lies in the cokernel $I \otimes_\bC \OB(\phi,p)$ of the map $(d_{V/C} s^\vee)^\vee$ in \eqref{long exact}. It is now easy to see using the diagram
\begin{align*}
\xymatrix{
g^* \Omega_{V /C} \vert_U \ar[r] \ar@{=}[d] & g^* E \ar[r] \ar[d]^-{g^* \phi} & g^* F^\vee |_U [1] \ar[dr]^-{(g')^*s^\vee|_\Delta} \ar[d]\\
g^* L_{V / C}^{\geq -1} \vert_U \ar[r] \ar[d] & g^* L_{U /C}^{\geq -1} \ar[d] \ar[r] & g^* L_{U/V}^{\geq -1} \ar[r] \ar[d]  & g^* J / J^2 [1] \ar[d] \\
L_{\bar{\Delta} / C}^{\geq -1} \vert_\Delta \ar[r] & L_{\Delta / C}^{\geq -1} \ar[r] & L_{\Delta / \bar{\Delta}}^{\geq -1} \ar[r] & I[1],
}
\end{align*}
that indeed $ob_U(\phi,g,\Delta,\bar{\Delta}) = \rho( s \circ g' )$.
\end{proof}

\begin{defi} \label{Same obs assign}
Let $\phi \colon E \to L_{U / C}^{\geq -1}$ and $\phi' \colon E' \to L_{U / C}^{\geq -1}$ be two perfect obstruction theories and $\psi \colon \Ob_\phi \to \Ob_{\phi'}$ be an isomorphism. We say that the obstruction theories give the same obstruction assignment via $\psi$ if for any infinitesimal lifting problem of $U/C$ at $p$
\begin{align}
\psi \left( ob_U(\phi,g,\Delta,\bar{\Delta}) \right) = ob_U(\phi',g,\Delta,\bar{\Delta}) \in \OB(\phi',p) \otimes_\bC I.
\end{align}
\end{defi}
We are now ready to give the definition of a semi-perfect obstruction theory.

\begin{defi} \emph{(Semi-perfect obstruction theory \cite{LiChang})} \label{semi-perfect obs th}
Let $\mM\to C$ be a morphism, where $\mM$ is a DM stack, proper over $C$, of finite presentation and $C$ is a smooth quasi-projective curve. A semi-perfect obstruction theory $\phi$ consists of an \'{e}tale covering $\lbrace U_\alpha \rbrace_{\alpha \in A}$ 
of $\mM$ and perfect obstruction theories $\phi_\alpha \colon E_\alpha \to L_{U_\alpha / C}^{\geq -1}$ such that
\begin{enumerate}
\item For each pair of indices $\alpha, \beta$, there exists an isomorphism \begin{align*}
\psi_{\alpha \beta} \colon \Ob_{\phi_\alpha} \vert_{U_{\alpha\beta}} \lra \Ob_{\phi_\beta} \vert_{U_{\alpha\beta}}
\end{align*}
so that the collection $\lbrace \Ob_{\phi\lalp}, \psi\lab \rbrace$ gives descent data of a sheaf on $\mM$.
\item For each pair of indices $\alpha, \beta$, the obstruction theories $E_\alpha \vert_{U_{\alpha \beta}}$ and $E_\beta \vert_{U_{\alpha \beta}}$ give the same obstruction assignment via $\psi_{\alpha \beta}$ (as in Definition \ref{Same obs assign}).
\end{enumerate}
\end{defi}

\begin{rmk}
The obstruction sheaves $\lbrace \Ob_{\phi\lalp} \rbrace_{\alpha \in A}$ glue to define a sheaf $\Ob_{\phi}$ on $\mM$. This is the obstruction sheaf of the semi-perfect obstruction theory $\phi$.
\end{rmk}

Suppose now that $\mM \to C$ is as above and admits a semi-perfect obstruction theory. Then, for each $\alpha \in A$, we have
\begin{align*}
\cC_{U\lalp/C} \subset N_{U\lalp/C} = h^1/h^0 ( (L_{U\lalp/C}^{\geq -1})^\vee ) \lhook\joinrel\xrightarrow{h^1/h^0(\phi\lalp^\vee)} h^1/h^0 ( E\lalp^\vee ) \lr h^1(E\lalp^\vee),
\end{align*}
where $\cC_{U\lalp/C}$ and $N_{U\lalp/C}$ denote the intrinsic normal cone stack and intrinsic normal sheaf stack respectively, where by abuse of notation we identify a sheaf $\fF$ on $\mM$ with its sheaf stack.

We therefore obtain a cycle class $[\fc_{\phi \lalp}] \in Z_* \Ob_{\phi\lalp}$ by taking the pushforward of the cycle $[\cC_{U\lalp/C}] \in Z_* N_{U\lalp/C}$. 

\begin{thm-defi} \emph{\cite[Theorem-Definition~3.7]{LiChang}} \label{cone}
Let $\mM$ be a DM stack, proper over $C$, of finite presentation and $C$ a smooth quasi-projective curve, such that $\mM \to C$ admits a semi-perfect obstruction theory $\phi$. The classes $[\fc_{\phi \lalp}] \in Z_* \Ob_{\phi\lalp}$ glue to define an intrinsic normal cone cycle $[\fc_\phi] \in Z_* \Ob_\phi$. Let $s$ be the zero section of the sheaf stack $\Ob_\phi$. The virtual cycle of $\mM$ is defined to be
\begin{align*}
[\mM, \phi]^{\mathrm{vir}} := s^{!} [\fc_\phi] \in A_* \mM,
\end{align*}
where $s^{!} \colon Z_* \Ob_\phi \to A_* \mM$ is the Gysin map. This virtual cycle satisfies all the usual properties, such as deformation invariance.
\end{thm-defi}

\begin{rmk} \label{5.10}
One can also consider \'{e}tale covers of $\mM$ by DM quotient stacks $[U\lalp / G\lalp]$, where $G\lalp$ acts on $U\lalp$ with finite stabilizers. There is a natural generalization of the notion of semi-perfect obstruction theory and Theorem~\ref{cone} in this setting. This will be used in Section~\ref{gendt} in order to glue the intrinsic normal cone cycles obtained by perfect obstruction theories on a cover of this form.
\end{rmk}

\section{Local Calculations} \label{LoCalc}

In this section, we establish terminology and collect a series of lemmas and propositions that will be useful in defining the semi-perfect obstructon theory of the intrinsic stabilizer reduction $\tilde{\mM}$ of a GIT quotient stack $\mM = [ X / G]$ with a d-critical structure $s \in \Gamma(\mM, \sS_\mM)$.

This is necessary as the Kirwan blowup of a d-critical quotient stack does not have to be d-critical, so we have to work with a local model that is more general than an invariant d-critical chart and whose structure is preserved by Kirwan blowups and \'{e}tale slices. At the same time, this local model should be equipped with a natural perfect complex which will provide us with a perfect obstruction theory when the Kirwan blowup procedure terminates. 

In addition, we define appropriate notions of compatibility between these local models, generalizing the concept of an embedding between d-critical charts, which we call $\Omega$-equivalence and $\Omega$-compatibility. $\Omega$-equivalence allows us to compare between local models with the same ambient smooth scheme, while $\Omega$-compatibility addresses the possibility of enlarging the ambient smooth scheme by adding extra coordinates. This comparison of local models is preserved under Kirwan blowups and \'{e}tale slices and implies suitable compatibilities between their associated perfect obstruction theories, which are essential in constructing the semi-perfect obstruction theory of $\tilde{\mM}$ in Section~\ref{gendt}.

\subsection{Local models and standard forms} \label{inductive step} Let $V$ be a smooth affine $G$-scheme. The action of $G$ on $V$ induces a morphism $\fg \otimes \oO_V \rightarrow T_V$ and its dual $\sigma_V : \Omega_V \rightarrow \fg^{\vee} \otimes \oO_V$. 

We consider the following data on $V$.

\begin{setup-def} \label{ind hyp}
The quadruple $(V, F_V, \omega_V, D_V)$, where $F_V$ is a $G$-equivariant vector bundle on $V$, $\omega_V$ an invariant section, with scheme theoretic zero locus $U = (\omega_V = 0) \sub V$, and $D_V \sub V$ an effective invariant divisor, satisfying:
\begin{enumerate}
\item $\sigma_V(-D_V)  : \Omega_V ( -D_V ) \to \fg^\vee (-D_V)$
factors through a morphism $\phi_V$ as shown
\begin{align} \label{5.23}
\Omega_V(-D_V) \lra F_V \xrightarrow{\phi_V} \fg^\vee(-D_V);
\end{align}
\item The composition $\phi_V \circ \omega_V$ vanishes identically;
\item Let $R$ be the identity component of the stabilizer group of a closed point in $V$ with closed orbit. Let $V^R$ denote the fixed point locus of $R$. Then $\phi_V |_{V^R}$ composed with the projection $\fg^\vee(-D_V) \rightarrow \fr^\vee(-D_V)$ is zero, where $\fr$ is the Lie algebra of $R$;
\end{enumerate}
gives rise to data $$\Lambda_V = (U, V, F_V, \omega_V, D_V, \phi_V)$$ on $V$. We say that these data give a weak local model structure for $V$. We also say that $U$ is in weak standard form.
\end{setup-def}

\begin{rmk}
Note that if $f \colon V \to \bA^1$ is a $G$-invariant function on $V$, then $(U, V, \Omega_V, df, 0, \sigma_V)$ give a weak local model for $V$, being equivalent to an invariant d-critical chart $(U,V,f,i)$ for $U$. Therefore, an invariant d-critical locus is a particular case of weak standard form.
\end{rmk}

The next lemma states that the structure of a weak local model behaves well under taking Kirwan blowups and slices thereof.

\begin{lem} \label{Obs bundle for blow} Let $\Lambda_V = (U, V, F_V,\omega_V,D_V, \phi_V)$ be as in Setup~\ref{ind hyp}. Let $\pi: \hV\to V$ be the Kirwan blowup of $V$ associated with $G$. Then 
we have induced data $\Lambda_{\hV} = (\hat U, \hat V, F_{\hV}, \omega_{\hV}, D_{\hV}, \phi_{\hV})$, where $F_{\hV}$ is the blowup bundle (cf. Definition~\ref{blow-up bundle}) of $F_V$, $\omega_{\hV}$ the blowup section and $D_{\hV}=\pi\sta D_V+2E$, that give a weak local model structure for $\hV$.

Moreover, for a slice $S$ of a closed point $x$ in ${\hV}$ with closed $G$-orbit and stabilizer $H$, we obtain induced data $\Lambda_S = (T, S, F_S, \omega_S, D_S, \phi_S)$, where $F_S$ is an $H$-equivariant bundle on $S$ with a section $\omega_S$ and the conditions of Setup~\ref{ind hyp} are satisfied for $S$ as well.
\end{lem}

\begin{proof} By pulling back via $\pi : \hat{V} \to V$ the factorization in \eqref{5.23}, we obtain
\begin{align*}
\pi^ *\Omega_V(-D_V) \lra \pi^* F_V \xrightarrow{\pi^* \phi_V} \fg^\vee(-D_V).
\end{align*} 

Let $E\sub\hV$ be the exceptional divisor. By slight abuse of notation, we use $D_V$ to also denote the pull-back of $D_V$ to the blow-up.
Then, applying (3) with  $R=G$, $\pi^* \phi_V$ factors through $\fg^\vee (-E-D_V)$. 
Using the obvious inclusion $\Omega_{\hat{V}}(-E) \rightarrow \pi^* \Omega_V$, we get that the morphism $\Omega_{\hat{V}}(-E-D_V) \rightarrow \fg^\vee (-E-D_V)$ induced by the action of $G$, factors as
\begin{align*}
\Omega_{\hat{V}} (-E-D_V) \lra \pi^* \Omega_V(-D_V) \lra \pi^* F_V \lra \fg^\vee (-E-D_V).
\end{align*}
We have the following diagram 
\begin{align}
\xymatrix{
\Omega_{\hat{V}}(-E-D_V) \ar[r] \ar[d] & \pi^* F_V \ar[r] \ar[d] & \fg^\vee (-E-D_V) \ar[d]\\ 
\Omega_{\hat{V}}(-E-D_V) |_E \ar[r] & \left( \pi^* F_V \right) |_E \ar[r] \ar[d] & \fg^\vee (-E-D_V) |_E\\
 & \pi^* \left( F_V |_{V^G}^{mv} \right) \ar[ur]
}
\end{align}
Note that 
\begin{align}
\left( \pi^* F_V \right) |_E \cong \pi^\ast \left( F_V |_{V^G}^{fix} \right) \oplus \pi^\ast \left( F_V |_{V^G}^{mv} \right)
\end{align}
and by equivariance we see that $\pi^* \phi_V |_E$ maps $\pi^\ast \left( F_V |_{V^G}^{fix} \right)$ to zero inside $\fg^\vee \left( -E-D_V \right) |_E$. This induces the last up-right arrow of the diagram.
Therefore by taking kernels, we get a factorization
\begin{align} \label{5.25}
\sigma_{\hat{V}} : \Omega_{\hat{V}}(-D_{\hat{V}}) \lra F_{\hat{V}} \xrightarrow{\phi_{\hat{V}}} \fg^\vee (-D_{\hat{V}}),
\end{align}
where $D_{\hat{V}} = 2E + D_V$. This shows (1) for $\hat{V}$.

By looking at the diagram, we can easily see that it follows from the identical vanishing of $\phi_V \circ \omega_V$ on $V$ that $\phi_{\hat{V}} \circ \omega_{\hat{V}}$ vanishes identically on $\hat{V}$. This is (2).

Let us check (3) on $\hat{V}$. Away from the exceptional divisor $E$, $\phi_{\hat{V}}$ is the same as $\phi_V$ and hence we have the same vanishing on $\hat{V}^R - E$. On the other hand, on $\hV$, Kirwan's general theory in \cite{Kirwan} guarantees that no new $R$ can arise from the blowup procedure and $(\hat{V})^R$ is the proper transform of $V^R$. It readily follows that, since $\phi_{\hat{V}} |_{(\hat{V})^R}$ composed with the projection onto $\fr^\vee (-D_{\hat{V}})$ is vanishing on $\hat{V} - E$, it vanishes on $(\hV)^R$ as desired.

Next we restrict \eqref{5.25} to a slice $S$ in $\hat{V}$. The fibration $G \times_H S \rightarrow G/H$ with fiber $S$ gives an exact sequence
\begin{align*}
0 \lra (\fg/ \fh)^\vee \lra \Omega_{\hat{V}} |_S \lra \Omega_S \lra 0,
\end{align*}
where $\fh$ is the Lie algebra of $H$. The composition of the first arrow $(\fg / \fh)^\vee \rightarrow \Omega_{\hat{V}} |_S$ with the homomorphism $\sigma_{\hat{V}} : \Omega_{\hat{V}} |_S \rightarrow \fg^\vee$ induced by the action of $G$ is the inclusion $(\fg / \fh)^\vee \hookrightarrow \fg^\vee$. Therefore $(\fg / \fh)^\vee (-D_{\hV}) |_S$ is a subbundle of 
$\Omega_{\hat{V}}(-D_{\hV})|_S$ and $\pi^* F_V |_S$ as well as $\fg^\vee(-D_{\hV})|_S$. If we take the quotient of \eqref{5.25} restricted to $S$ by $(\fg / \fh)^\vee (-D_{\hV})|_S$, we obtain a factorization 
\begin{align}
\sigma_S : \Omega_S(-D_S) \lra F_S \xrightarrow{\phi_S} \fh^\vee(-D_S)
\end{align}
of the morphism $\sigma_S$ induced by the action of $H$ on $S$, where $D_S$ is the restriction of $D_{\hV}$ to $S$. 
This shows (1) for $S$. Finally, it is not hard to verify that (2) and (3) are also true. \end{proof}

We may now give the following definition, which will be useful when we introduce the concept of $\Omega$-compatibility later in the paper.
\begin{defi} \label{local model}
We say that the data $\Lambda_V = (U, V, F_V, \omega_V, D_V, \phi_V)$ give a local model for $V$ if either they are the data of a d-critical chart or are obtained by such after a sequence of Kirwan blowups and/or taking slices of closed points with closed orbit. We also say that $U$ is then in standard form.
\end{defi}

\subsection{Obstruction theory of local model} \label{red obs th sec}
Suppose that $U$ is in weak standard form for data $\Lambda_V$ of a local model on $V$.

\begin{lem} \label{lemma complex}
The following sequence is a complex:
\begin{align}  \label{complex}
K_V = [ \fg \stackrel{\sigma_V^\vee}{\lra} T_V \vert_U \xrightarrow{\left( d_V \omega_V^\vee \right)^\vee} F_V \vert_U \stackrel{\phi_V}{\lra} \fg^\vee(-D_V) ].
\end{align}
\end{lem}

\begin{proof}
Since $\omega_V$ is $G$-invariant, the composition $\left( d_V \omega_V^\vee \right)^\vee \circ \sigma_V=0$. Moreover, 
since $\phi_V \circ \omega_V=0$, by differentiating we obtain $\phi_V \circ \left( d_V \omega_V^\vee \right)^\vee = - \omega_V^\vee \circ d_V\phi_V$, which is zero when restricted to $U=(\omega_V=0)$. This proves the lemma.
\end{proof}

\begin{defi} \emph{(Reduced tangent and obstruction sheaf)}
We define the reduced tangent sheaf and reduced obstruction sheaf of $V$ to be
$$T_V^\rred := \coker{\sigma_V^\vee}, \quad \text{and}\quad
F_V^\rred := \ker{\phi_V}.
$$
The section $\omega_V$ induces a section of $F_V^\rred$, denoted by $\omega_V^\rred$.
\end{defi}

The restriction of the complex \eqref{complex} to the stable part $U^s$ of $U$ gives rise (and is quasi-isomorphic) to a two-term complex
\begin{align*}
K_V^\rred = [\bigl( d_V (\omega_V^\rred)^\vee \bigr)^\vee: T_V^\rred \vert_{U^s} \lra F_V^\rred \vert_{U^s} ].
\end{align*}
We denote $\coker \bigl( d_V (\omega_V^\rred)^\vee \bigr)^\vee |_{U^s}$ by $\Ob_V^\rred$. One can easily check that there are natural isomorphisms
\begin{align} \label{H^1 to Ob}
H^1(K_V) |_{U^s} \cong H^1(K_V|_{U^s}) \cong \Ob_{V}^\rred .
\end{align} 

By slight abuse of terminology, we also refer to $\Ob_V^\rred$ as the reduced obstruction sheaf. This is validated by the following proposition.

\begin{prop} \emph{(Reduced obstruction theory)}
The dual of $K_V^\rred$ induces a perfect obstruction theory on the DM stack $\left[ U^s / G\right]$.
\end{prop}

\begin{proof}
On $V^s$, $\sigma_V^\vee$ is injective and $\phi_V$ is surjective. The latter follows from the fact that the surjective morphism $\sigma_V^\vee (-D_V)$ factors through $\phi_V$. In particular, the two terms of $K_V^\rred$ are vector bundles and $U^s$ is the zero locus of $\omega_V^\rred \vert_{V^s}$.

Let $q_V \colon V^s \to \left[ V^s / G \right]$ be the quotient morphism. We have the exact triangle of truncated cotangent complexes
\begin{align*}
q_V^* L_{\left[ V^s / G \right]}^{\geq -1} \lr L_{V^s}^{\geq -1} \lr \fg^\vee \lr q_V^* L_{\left[ V^s / G \right]}^{\geq -1}[1],
\end{align*}
from which we deduce that $T_V^\rred |_{V^s}= q_V^* T_{\left[ V^s / G \right]}^{\leq 1}$. Therefore, using the same triangle for $U^s$, we find that
\begin{align*}
q_U^* L_{\left[ U^s / G \right]}^{\geq -1} = [ I^s / (I^s)^2 \lra \Omega_V^\rred |_{U^s} ],
\end{align*}
where $I^s$ is the ideal of $U^s$ in $V^s$, $\Omega_V^\rred|_{V^s} = \left( T_V^\rred |_{V^s} \right)^\vee = \ker \sigma_V |_{V^s} = q_V^* \Omega_{\left[ V^s / G \right]} = q_V^* L_{\left[ V^s / G \right]}^{\geq -1}$ and that there is an arrow
\begin{align*}
\xymatrix{
K_V^\rred|_{U^s}^\vee \ar@{=}[r] \ar[d] & [ F_V^\rred \vert_{U^s}^\vee \ar[r]^-{d_V (\omega_V^\rred)^\vee} \ar[d]^-{(\omega_V^\rred)^\vee} & \Omega_V^\rred \vert_{U^s} \ar@{=}[d] ]\\
q_U^* L_{\left[ U^s / G \right]}^{\geq -1} \ar@{=}[r] & [ I^s/(I^s)^2 \ar[r]^-{d_{V}} & \Omega_V^\rred \vert_{U^s} ].
}
\end{align*} 
Therefore $K_V^\rred|_{U^s}$ descends to a perfect obstruction theory on $\left[ U^s / G \right]$.
\end{proof}

\begin{rmk}
Since the rank of $F_V$ is equal to $\dim V$, in order to obtain a zero-dimensional virtual cycle, we need to replace $F_V|_{U^s}$ by $F_V^\rred|_{U^s}$, which has rank $\dim V - \dim G = \dim [V^s / G]$.
Note  that the surjective morphism $F_V |_{U^s} \stackrel{\phi_V}{\to} \fg^\vee(-D_V)$ 
induces a twisted cosection $\Ob_{U^s} \to \fg^\vee(-D_V)$ (cf. \cite{cosection}), which enables us to make the perfect obstruction theory $0$-dimensional.
\end{rmk}

\subsection{$\Omega$-equivalence} We introduce the following definition.

\begin{defi} \emph{($\Omega$-equivalence)} \label{omega equivalence}
Let $V$ be a smooth affine $G$-scheme and $F_V$ a $G$-equivariant bundle on $V$. We say that two invariant sections $\omega_V, \bom_V \in \Gamma(V, F_V)$ are $\Omega$-equivalent if
\begin{enumerate}
\item $(\omega_V) = (\bom_V) =: I_U$ as ideal sheaves in $\oO_V$, and
\item there exist equivariant morphisms $A,B \colon F_V \to T_V$ such that
\begin{align}
\label{omegaf - omegag} \omega_V^\vee = \bom_V^\vee + \bom_V^\vee \circ A^\vee \circ (d\bom_V^\vee) \ \left( \text{mod} \ I_U^2 \right), \\
\label{omegag - omegaf} \bom_V^\vee = \omega_V^\vee + \omega_V^\vee \circ B^\vee \circ (d\omega_V^\vee) \ \left( \text{mod} \ I_U^2 \right).
\end{align}
\end{enumerate}
\end{defi}

The reason for introducing this notion is the following proposition.

\begin{prop} \label{equivalence for df and dg}
Let $U$ be an affine $G$-invariant d-critical scheme and $(U,V,f,i)$ and $(U,V,g,i)$ two invariant d-critical charts with $V$ an affine, smooth $G$-scheme. Then $\omega_f = df, \ \omega_g = dg$ are $\Omega$-equivalent sections of $\Omega_V$.
\end{prop}

\begin{proof} We may assume that $f-g \in I_U^2$, where $I_U$ is the ideal of $U$ in $V$.

Let $x_1, ..., x_n$ be \'{e}tale coordinates on $V$. Let us write $f_i$ for $\frac{\partial f}{\partial x_i}$, $f_{ij}$ for $\frac{\partial^2f}{\partial x_i \partial x_j}$ and $H_f$ for the Hessian of $f$ for convenience (and similarly for $g$). Then we have $I_U = ( f_i )_{i=1}^n = ( g_ i )_{i=1}^n$. Moreover $f - g \in I_U^2$ implies that 
\begin{align} \label{f-g}
f = g + \sum_{k,l} a^{kl} \cdot g_k \cdot g_l \text{, } a^{kl} \in \Gamma(\oO_V).
\end{align}

Differentiating, we obtain for any $i$ and pair $(i,j)$ the relations
\begin{align} \label{df - dg}
f_i = g_i + \sum_{k,l} a^{kl} \cdot g_{ki} \cdot g_l + \sum_{k,l} a^{kl} \cdot g_k \cdot g_{li} \text{ mod } I_U^2, \\ \notag
f_{ij} = g_{ij} + \sum_{k,l} a^{kl} \cdot g_{lj} \cdot g_{ki} + \sum_{k,l} a^{kl} \cdot g_{kj} \cdot g_{li} \text{ mod } I_U.
\end{align}
Note that we may re-write \eqref{df - dg} as 
\begin{align*}
\omega_f = \omega_g + H_g \circ A \circ dg \left( \text{mod} \ I_U^2 \right),
\end{align*}
where $A$ is a morphism $A \colon \Omega_V \rightarrow T_V$.

Since $df$, $dg$ and $H_g$ are invariant, applying the Reynolds operator we can assume that $A$ is equivariant.
Hence
\begin{align} 
\omega_f^\vee = \omega_g^\vee + \omega_g^\vee \circ A^\vee \circ (d_V\omega_g^\vee)  \ \left( \text{mod} \ I_U^2 \right),
\end{align}
and similarly for $g$ we have
\begin{align} 
\omega_g^\vee = \omega_f^\vee + \omega_f^\vee \circ B^\vee \circ (d_V\omega_f^\vee)  \ \left( \text{mod} \ I_U^2 \right).
\end{align}
This proves the proposition. \end{proof}

The following two lemmas show that $\Omega$-equivalence is preserved by the operations of Kirwan blowup and taking slices of closed orbits.

\begin{notation} In what follows $$\Lambda_V = (U,V,F_V,\omega_V, D_V,\phi_V), \ \bar{\Lambda}_V = (U,V,F_V,\bom_V, D_V,\phi_V)$$ will denote data of a weak local model structure on $V$. Similarly on the Kirwan blowup we write $$\Lambda_{\hV} = (\hU, \hV, F_{\hV}, \omega_{\hV}, D_{\hV}, \phi_{\hV}), \ \bar{\Lambda}_{\hV} = (\hU, \hV, F_{\hV}, \bom_{\hV}, D_{\hV}, \phi_{\hV})$$ and on an \'{e}tale slice thereof $$\Lambda_S = (T,S,F_S,\omega_S, D_S,\phi_S),\ \bar{\Lambda}_S = (T,S,F_S,\bom_S, D_S,\phi_S).$$
\end{notation}

\begin{lem}\label{comparison} Let $V$ be a smooth affine $G$-scheme and $\Lambda_V$, $\bar{\Lambda}_V$ as above, such that $\omega_V, \bom_V$ are $\Omega$-equivalent with $A, B$ as in \eqref{omegaf - omegag} and \eqref{omegag - omegaf}. Then the blowup sections $\omega_{\hV}, \bom_{\hV}$ are $\Omega$-equivalent (via induced equivariant morphisms $\hat{A}, \hat{B} \colon F_{\hV} \to T_{\hV}$).
\end{lem}

\begin{proof}
Let $i_E : E \to \hat{V}$ be the exceptional divisor of the blowup $\pi \colon \hV \to V$.
We have $N_{V^G/V} = T_V \vert_{V^G}^{mv}$ and therefore the relative Euler sequence
\begin{align} \label{5.5}
0 \lra \oO_E(E) \lra \pi^* \left( T_V \vert_{V^G}^{mv} \right) \stackrel{\alpha}{\lra} T_{E/V^G}(E) \lra 0.
\end{align}
We also have the tangent sequence
\begin{align} \label{5.6}
0 \lra T_{\hat{V}} \stackrel{d\pi}{\lra} \pi^* T_V \stackrel{\beta}{\lra} i_{E *} T_{E/V^G}(E) \lra 0.
\end{align}
Using these and the definition of $F_{\hat{V}}$ we have the following commutative diagram
\begin{align} \label{5.7}
\xymatrix{
0 \ar[r] & F_{\hat{V}} \ar[r]^-\gamma \ar@{-->}[d]^-{\hat{A}} & \pi^* F_V \ar[r] \ar[d]^-{\pi^* A} & \pi^* F_V \vert_E \ar[r] \ar[d] & \pi^* \left( F_V \vert_{V^G}^{mv} \right) \ar[d] \\
0 \ar[r] & T_{\hat{V}} \ar[r]^-{d\pi} & \pi^* T_V \ar[r] \ar@<-0.3ex>[drr]_\beta & \pi^* T_V \vert_E \ar[r] & \pi^* \left( T_V \vert_{V^G}^{mv} \right) \ar[d]^-\alpha \\
& & & & T_{E / V^G}(E).\\
}
\end{align}
By \eqref{5.5}, \eqref{5.6} and \eqref{5.7} it follows that the composition $\beta \circ \pi^* A \circ \gamma$ is zero and therefore we obtain an equivariant morphism $\hat{A} \colon F_{\hV} \to T_{\hV}$ induced by $\pi^* A$. 

It remains to check that $\hat{A}$ satisfies \eqref{omegaf - omegag} for the blowup sections ${\omega}_{\hV}$ and $\bom_{\hV}$. We can dualize and pull back \eqref{omegaf - omegag} to obtain
\begin{align} \label{5.8}
\pi^* \omega_V = \pi^* \bom_V +  \pi^* \left( d_V \bom_V^\vee \right)^\vee \circ \pi^* A \circ \pi^* \bom_V \left( \text{mod} \ \pi^* I_U^2 \right).
\end{align}
From the diagram
\begin{align}
\xymatrix{
\oO_{\hV} \ar[r]^-{\bom_{\hV}} \ar[dr]_{\pi^* \bom_{V}} & F_{\hV} \ar[r]^{\hat{A}} \ar[d]^-{\gamma} & T_{\hV} \ar[r]^-{\left( d_{\hV} \bom_{\hV}^\vee \right)^\vee} \ar[d]^-{d\pi} & F_{\hV} \ar[d]^-{\gamma}\\
& \pi^* F_V \ar[r]_{\pi^* A} & \pi^* T_V \ar[r]_{\pi^* \left( d_V \bom_{V}^\vee \right)^\vee} & \pi^* F_V 
}
\end{align}
we see that
\begin{align}
\pi^* \left( d_V \bom_{V}^\vee \right)^\vee \circ \pi^* A \circ \pi^* \bom_{V} = \gamma \circ \left( d_{\hV} \bom_{\hV}^\vee \right)^\vee \circ \hat{A} \circ \bom_{\hV} . 
\end{align}
Therefore, we may re-write \eqref{5.8} as
\begin{align}
\gamma \circ \left( \omega_{\hV} - \bom_{\hV} - \left( d_{\hV} \bom_{\hV}^\vee \right)^\vee \circ \hat{A} \circ \bom_{\hV} \right) \in \Gamma \left( \hV, \pi^* I_U^2 \cdot \pi^* F_V \right).
\end{align}
By equivariance, we can un-twist by $\gamma$ (which is multiplication by $\xi$, the equation of the exceptional divisor, on the moving part of $\pi^* F_V$) to obtain
\begin{align}
\omega_{\hV} = \bom_{\hV} + \left( d_{\hV} \bom_{\hV}^\vee \right)^\vee \circ \hat{A} \circ \bom_{\hV} \left( \text{mod} \ I_{\hat{U}}^2 \right).
\end{align}
Here we are using the fact that if $\eta$ is an invariant section of $I_U^2 \cdot F_V$, then the induced section $\hat{\eta}$ factors through $I_{\hU}^2 \cdot F_{\hV}$.

Dualizing the latter, we obtain relation \eqref{omegaf - omegag} for the blowup. \eqref{omegag - omegaf} for the blowup follows by the same argument.
\end{proof}

\begin{lem} \label{comparison for slice}
Let $S$ be an \'{e}tale slice of a closed point $x \in \hat{V}$ with closed orbit and stabilizer $H$. Let $\Lambda_{\hV}$, $\bar{\Lambda}_{\hV}$ be data of a weak local model on $\hV$ such that $\omega_{\hV}, \bom_{\hV}$ are $\Omega$-equivalent. Let $\omega_S, \bom_S$ be the two sections obtained as part of the local setup induced on $S$ using these two choices of data and Lemma~\ref{Obs bundle for blow}. 
Then $\omega_S, \bom_S$ are also $\Omega$-equivalent.
\end{lem}

\begin{proof}
Let $\hat{U}$ be the zero locus of $\bom_{\hV}$ (or equivalently $\omega_{\hV}$) and $T$ be the zero locus of $\bom_S$ (equivalently $\omega_S$). 
The commutative diagram
\begin{align*}
\xymatrix{
0 \ar[r] & \left(\fg / \fh \right)^\vee \ar@{=}[d] \ar[r] & \Omega_{\hV} |_S \ar[r]^-{q_S} \ar[d]^-{\sigma_{\hV}^\vee} & \Omega_S \ar[r] \ar[d]^-{\sigma_S^\vee} & 0\\
0 \ar[r] & \left(\fg / \fh \right)^\vee \ar[r] & \fg^\vee \ar[r] & \fh^\vee \ar[r] & 0
}
\end{align*}
induces an isomorphism of kernels $\Omega_{\hV} \vert_T^\rred \xrightarrow{q_S|_T} \Omega_S|_T^\rred$ of $\sigma_{\hV}^\vee|_T$ and $\sigma_S^\vee|_T$ respectively. Since both exact sequences are locally split we may shrink $S$ around $x$ and find a ($H$-equivariant) right inverse $r_S \colon \Omega_S \to \Omega_{\hV}|_S$ for $q_S$, which is then also an inverse for $q_S|_T$ when restricted to $\Omega_S|_T^\rred$.

By the same argument for the commutative diagram (recall the definition of $F_S$ in the proof of Lemma~\ref{Obs bundle for blow})
\begin{align*}
\xymatrix{
0 \ar[r] & \left(\fg / \fh \right)^\vee(-D_S) \ar@{=}[d] \ar[r] & F_{\hV} |_S \ar[r]^-{\gamma_S} \ar[d]^-{\phi_{\hV}} & F_S \ar[r] \ar[d]^-{\phi_S} & 0\\
0 \ar[r] & \left(\fg / \fh \right)^\vee(-D_S) \ar[r] & \fg^\vee(-D_S) \ar[r] & \fh^\vee(-D_S) \ar[r] & 0
}
\end{align*}
we may find a right inverse $\delta_S \colon F_S \to F_{\hV}|_S$, which is an inverse for $\gamma_S$ when restricted to $F_S|_T^\rred$, the kernel of $\phi_S|_T$.

Consider the (non-commutative) diagram
\begin{align*}
\xymatrix{
F_{\hV}^\vee |_T \ar[r]^-{d_{\hV} \bom_{\hV}^\vee} & \Omega_{\hV}|_T \ar[d]^-{q_S} \ar[r]^-{\hat{A}^\vee} & F_{\hV}^\vee|_T \ar[r]^-{\bom_{\hV}^\vee} & I_{\hU} / I_{\hU}^2|_T \ar[d]^-{q_S}\\
F_S^\vee |_T \ar[r]_-{d_S \bom_S^\vee} \ar[u]^-{\gamma_S^\vee} & \Omega_S|_T \ar@{-->}[r]_-{\hat{A}_S^\vee} & F_S^\vee|_T \ar[u]^-{\gamma_S^\vee} \ar[r]_{\bom_S^\vee} & I_T / I_T^2,
}
\end{align*}
where we define $\hat{A}_S^\vee := \delta_S^\vee \circ \hat{A}^\vee \circ r_S$. The left- and rightmost squares are commutative.

We check now that the composition $q_S \circ \bom_{\hV}^\vee \circ \hat{A}^\vee \circ d_{\hV} \bom_{\hV}^\vee \circ \gamma_S^\vee$ is equal to $\bom_S^\vee \circ \hat{A}_S^\vee \circ d_S \bom_S^\vee$. 

By Lemma~\ref{Obs bundle for blow} and Lemma~\ref{lemma complex}, the above diagram factors through the sub-diagram
\begin{align*}
\xymatrix{
\Omega_{\hV}|_T^\rred \ar[r] \ar[d]^-{q_S} & \Omega_{\hV}|_T \ar[d]^-{q_S} \ar[r]^-{\hat{A}^\vee} & F_{\hV}^\vee|_T \ar[r] & \left( F_{\hV} |_T^\rred \right)^\vee \\
\Omega_S |_T^\rred \ar[r] & \Omega_S|_T \ar@{-->}[r]_-{\hat{A}_S^\vee} & F_S^\vee|_T \ar[u]^-{\gamma_S^\vee} \ar[r] & \left( F_S |_T^\rred \right)^\vee \ar[u]^-{\gamma_S^\vee}.
}
\end{align*}
Observe now that by the definition of $\hat{A}_S^\vee$ the outer square in this diagram commutes. This immediately implies that indeed
\begin{align} \label{6.13}
q_S \circ \bom_{\hV}^\vee \circ \hat{A}^\vee \circ d_{\hV} \bom_{\hV}^\vee \circ \gamma_S^\vee = \bom_S^\vee \circ \hat{A}_S^\vee \circ d_S \bom_S^\vee.
\end{align}
Since $\omega_S^\vee = q_S \circ \omega_{\hV}^\vee \circ \gamma_S^\vee$ and $\bom_S^\vee = q_S \circ \bom_{\hV}^\vee \circ \gamma_S^\vee$, applying $q_S \circ (\bullet) \circ \gamma_S^\vee$ to the restriction of \eqref{omegaf - omegag} to the slice $S$ and using \eqref{6.13} we obtain
\begin{align*}
\omega_S^\vee = \bom_S^\vee + \bom_S^\vee \circ \hat{A}_S^\vee \circ d_S \bom_S^\vee \left( \text{mod} \ I_T^2 \right) .
\end{align*}
The exact same argument can be used to show the existence of $\hat{B}_S$.
\end{proof}

The next lemma states that for the purposes of comparing obstruction sheaves and assignments we may replace a section by any $\Omega$-equivalent section without any effect.

Given two $\Omega$-equivalent $\omega_V, \bom_V \in H^0(F_V)$, denoting $U = ( \omega_V = 0 ) = ( \bom_V = 0 )$, we obtain

$$\left( d \omega_V^\vee \right)^\vee \vert_U, \ \left( d \bom_V^\vee \right)^\vee \vert_U \colon T_V|_U \lra  F_V |_U.$$

\begin{lem} \label{Obstruction sheaf comparison}
Let $V$ be a smooth affine $G$-scheme and $\Lambda_V$, $\bar{\Lambda}_V$ data of a weak local model on $V$, such that $\omega_V, \bom_V$ are $\Omega$-equivalent. Then
\begin{align} \label{6.21}
\coker \left( d \omega_V^\vee \right)^\vee \vert_{U} = \coker \left( d \bom_V^\vee \right)^\vee \vert_{U}.
\end{align}
Moreover, the two obstruction theories on $U$ induced by $\omega_V$ and $\bom_V$ give the same obstruction assignments via the morphism \eqref{6.21}. 
\end{lem}

\begin{proof}
We check that $\im \left( d \omega_V^\vee \right)^\vee \vert_{U} = \im \left( d \bom_V^\vee \right)^\vee \vert_{U}$. Since $\omega_V, \bom_V$ are $\Omega$-equivalent, there exist equivariant morphisms $A,B \colon F \to T_V$ such that
\begin{align}\label{5.14}
\omega_V^\vee = \bom_V^\vee + \bom_V^\vee \circ A^\vee \circ (d\bom_V^\vee) \ \left( \text{mod} \ I_U^2 \right), \\
\notag \bom_V^\vee = \omega_V^\vee + \omega_V^\vee \circ B^\vee \circ (d\omega_V^\vee) \ \left( \text{mod} \ I_U^2 \right).
\end{align}
Differentiating \eqref{5.14} and dualizing, we obtain
\begin{align}
\left( d \omega_V^\vee \right)^\vee = \left( d \bom_V^\vee \right)^\vee +  \left( d \bom_V^\vee \right)^\vee \circ A \circ \left( d \bom_V^\vee \right)^\vee \left( \text{mod} \ I_{{U}} \right).
\end{align}
This implies that $\im \left( d \omega_V^\vee \right)^\vee|_U \subseteq \im \left( d \bom_V^\vee \right)^\vee|_U$. By the same argument, using the second equation in \eqref{5.14}, the first claim follows.

For the obstruction assignments, consider an infinitesimal lifting problem of $U$ at $p$. Let $\Ob = \coker \left( d \omega_V^\vee \right)^\vee \vert_{U}$ and 
$$\rho \colon I \otimes_\bC F_V|_p \to I \otimes_\bC H^1(E|_p) = I \otimes_\bC H^1(\bar{E}|_p)$$
be the quotient morphism.  
Then by Lemma~\ref{compute obs assign}, we need to show that $\rho( \omega_V \circ g') = {\rho} ( \bom_V \circ g')$. But this holds, since dualizing \eqref{5.14} and composing with $g'$ we have that $\omega_V \circ g' - \bom_V \circ g' \in I \otimes_\bC \im \left( d \omega_V^\vee \right)^\vee|_p$.
\end{proof}

\subsection{$\Omega$-compatibility} 

We begin with the following lemma, describing how normal bundles behave with respect 
to Kirwan blowups. It will be used repeatedly in the rest of this subsection.
\begin{lem} \label{conormal transform}
Let $\Phi \colon V \hookrightarrow W$ be a $G$-equivariant embedding of smooth affine schemes. Let $N_{V/W}$ be the normal bundle of $V$ in $W$, and let $\mathrm{bl}(N_{V/W})$ as in Definition~\ref{blow-up bundle}.
Then there is a natural isomorphism $\mathrm{bl}(N_{V/W}) \cong N_{\hV/\hat{W}}$.
\end{lem}
\begin{proof}
Let $x \in V^G$. Up to shrinking, we have a commutative diagram
\begin{align*}
\xymatrix{
V \ar[d]^-{\Phi} \ar[r] & T_x V = \bA^n \ar[d]^-{d\Phi}\\
W \ar[r] & T_x W = \bA^{n+m},
}
\end{align*}
where the maps $V \to T_x V, \ W \to T_x W$ are equivariant \'{e}tale and the $G$-action on the tangent spaces is linear. Since $G$ is reductive, we may pick coordinates $x_1, \dots, x_n$ on $\bA^n$ on which $G$ acts linearly and extra coordinates $x_{n+1}, \dots, x_{n+m}$ with a linear $G$-action on $\bA^{n+m}$ such that the embedding $\bA^n \to \bA^{n+m}$ takes the canonical form $(x_1, \dots, x_n) \to (x_1, \dots, x_n, 0, \dots, 0)$. In particular, we get \'{e}tale coordinates $x_1|_W, \dots, x_{n+m}|_W$ on $W$, $x_1|_V, \dots, x_n|_V$ on $V$ and may also arrange that $I_V = ( x_{n+1}|_W, \dots, x_{n+m}|_W )$.

In what follows, we often write just $x_i$ in place of $x_i |_W$ or $x_i |_V$ by abuse of notation.

Let us assume that $x_1, \dots, x_p$ and $x_{n+1}, \dots, x_{n+q}$ are moving and $x_{p+1}, \dots x_{n}$ and $x_{n+q+1}, \dots, x_{n+m}$ are fixed by $G$.

Since the question is local, we may localize at $x$ and assume that $V$ and $W = \Spec A$ are local and the maximal ideal of $A$ is $$\fm = ( x_1, \dots, x_p, x_{n+1}, \dots, x_{n+q} ).$$

Now $\hat{W}$ is covered by open affines of the form (for $k=1,\dots,p+q$)
\begin{align}
R_k^n = \Spec A[T_{k,1}, \dots, T_{k,p+q}][\xi_k] / (T_{k,k} - 1),
\end{align}
where $x_i = \xi_k T_{k,i}$ if $i \leq p$, $x_{n-p+i} = \xi_k T_{k,i}$ if $i >p$ and $\xi_k=0$ is the exceptional divisor in $R_k^n$.
It is easy to see that $\hV$ is covered by such affines for $k \leq p$ and in each such we have that 
\begin{align*} 
I_{\hV}|_{R^n_k} & = \left( \frac{1}{\xi_k} x_{n+1}, \dots, \frac{1}{\xi_k} x_{n+q}, x_{n+q+1}, \dots, x_{n+m} \right) \\
 & = \left( T_{k,p+1}, \dots, T_{k,p+q}, x_{n+q+1}, \dots, x_{n+m} \right).
\end{align*}
In particular, $N_{\hV / \hat{W}} \vert_{R_k^n}$ has a basis of sections given by $ \frac{\partial}{\partial T_{k,p+i}}, \frac{\partial}{\partial x_{n+q+j}}$.

Since $N_{V/W}$ has a basis by $\frac{\partial}{\partial x_{n+i}}$, we see that $\mathrm{bl}( N_{V/W} )$ has a basis by ${\xi_k} \frac{\partial}{\partial x_{n+i}}$ for $1 \leq i \leq q$ and $\frac{\partial}{\partial x_{n+j}}$ for $q+1 \leq j \leq m$. But $x_{n+i} = \xi_k T_{k,p+i}$ implies that $\frac{1}{\xi_k} d_{\hat{W}} x_{n+i} = d_{\hat{W}} T_{k,p+i} \ \left( \text{mod } I_{\hV} \right)$ and thus by dualizing ${\xi_k} \frac{\partial}{\partial x_{n+i}} = \frac{\partial}{\partial T_{k,p+i}} \ \left( \text{mod } I_{\hV} \right)$. We see that the frames of the two bundles $N_{\hV/\hat{W}}$ and $\mathrm{bl}( N_{V/W} )$ match. This concludes the proof. \end{proof}

For the purposes of comparing obstruction theories obtained by embeddings of d-critical charts and their Kirwan blowups and slices thereof, we introduce the following definition.

\begin{defi} \emph{($\Omega$-compatibility)} \label{omega compatibility} 
Let $U \sub V \xrightarrow{\Phi} W$ be a sequence of $G$-equivariant embeddings of affine $G$-schemes such that $U$ is in standard form for data $\Lambda_V = (U, V, F_V, \omega_V, D_V, \phi_V)$ and $\Lambda_W = (W, F_W, \omega_W, D_W, \phi_W)$. We say that $\omega_V$ and $\omega_W$ are $\Omega$-compatible via $\Phi$ if the following hold:
\begin{enumerate}
\item $D_W$ pulls back to $D_V$ under $\Phi$.
\item The embedding $\Phi \colon V \to W$ induces a surjective equivariant morphism $\eta_\Phi \colon F_W|_V \to F_V$, compatible with $\phi_W|_V$ and $\phi_V$, such that $\eta_\Phi (\omega_W|_V)$ is $\Omega$-equivalent to $\omega_V$. 
\item Let  $\Ob_W := \coker \left( d_W \omega_W^\vee \right)^\vee \vert_{U}, \ \Ob_V := \coker \left( d_V \omega_V^\vee \right)^\vee \vert_{U}$. Then $\eta_\Phi$ induces an isomorphism 
\begin{align*}
\Phi^\ck \colon \Ob_W \lra  \Ob_V.
\end{align*}
\end{enumerate}
\end{defi}

\begin{rmk}
It makes sense to talk about an induced surjection $\eta_\Phi$, since the local data we are considering arise either as the data of an embedding of d-critical charts or are obtained by such an embedding by performing Kirwan blowups and taking slices. In the former case $\eta_\Phi$ is just pullback of differential forms and in the latter it is canonically induced starting from pulling back differential forms and then blowing up or taking slices.
\end{rmk}

The motivation behind the above definition is the following lemma.

\begin{lem} \label{omega compat for d-crit}
Let $(U,V,f,i) \xrightarrow{\Phi} (R,W,g,j)$ be a $G$-equivariant embedding of invariant d-critical charts, where $V\sub W$ as before is a pair of smooth
$G$-schemes. Then $\omega_g = dg$ and $\omega_f = df$ are $\Omega$-compatible. 
\end{lem}

\begin{proof} 
$\eta_\Phi$ is pullback of differential forms and $\eta_\Phi( \omega_g|_V ) = \omega_f$. Moreover, we have exact sequences
\begin{align*}
\xymatrix{
T_W|_U \ar[r]^-{\left( d_W\omega_g^\vee \right)^\vee} & \Omega_W|_U \ar[d]^-{\eta_\Phi} \ar[r] & \Omega_U \ar@{=}[d] \ar[r] & 0 \\
T_V|_U \ar[r]^-{\left( d_V\omega_f^\vee \right)^\vee} & \Omega_V|_U \ar[r] & \Omega_U \ar[r] & 0 
}
\end{align*}
from which we deduce that $\coker \left( d_W\omega_g^\vee \right)^\vee |_U = \coker ( d_V\omega_f^\vee )^\vee |_U = \Omega_U$. We obtain an isomorphism on cokernels induced naturally by $\eta_\Phi$.
\end{proof}

In the rest of this subsection, we consider the following situation:

\begin{notation}
Let $U \sub V \xrightarrow{\Phi} W$ be a sequence of $G$-equivariant embeddings of affine $G$-schemes such that $U$ is in standard form for data $\Lambda_V = (U, V, F_V, \omega_V, D_V, \phi_V)$ and $\Lambda_W = (U, W, F_W, \omega_W, D_W, \phi_W)$ and $\omega_V$ and $\omega_W$ are $\Omega$-compatible.
\end{notation}

We now check that one may compare obstruction theories given by different data of a local model with $\Omega$-compatible sections.

\begin{lem} \label{obs assign for omega compat}
Let $\Lambda_V$, $\Lambda_W$ be as above. Consider the two complexes $E_W, E_V$
\begin{align*}
(L_U^{\geq -1})^\vee \lra & E_W = [ T_W |_U \xrightarrow{(d_W \omega_W^\vee)^\vee} F_W |_U ], \\
(L_U^{\geq -1})^\vee \lra & E_V = [ T_V |_U \xrightarrow{(d_V \omega_V^\vee)^\vee} F_V |_U ],
\end{align*}
on $U$.
The obstruction theories induced by the dual complexes $E_W^\vee, E_V^\vee$ give the same obstruction assignments via $\Phi^\ck$.
\end{lem}

\begin{proof} 
Suppose we have an infinitesimal lifting problem at a closed point $p \in U$. Let 
$$\rho_W \colon I \otimes F_W|_p \to I \otimes \Ob_W |_p\quad \text{and}\quad \rho_V \colon I \otimes F_V|_p \to I \otimes \Ob_V |_p$$ be the induced quotient morphisms. By Lemma~\ref{compute obs assign}, we need to show 
\begin{align} \label{6.30}
\Phi^\ck \left( \rho_W ( \omega_W \circ \Phi \circ g') \right) = \rho_V ( \omega_V \circ g')
\end{align}
Since $\eta_\Phi(\omega_W|_V)$ and $\omega_V$ are $\Omega$-equivalent, by Lemma~\ref{Obstruction sheaf comparison}, we may assume that $\eta_\Phi(\omega_W|_V) = \omega_V$.

Because of $\eta_\Phi \circ \omega_W \circ \Phi = \omega_V$, $\Phi^\ck \circ \rho_W = \rho_V \circ \eta_\Phi$, and the commutative diagram
\begin{align*}
\xymatrix{
\bar{\Delta} \ar[r]^-{g'} & V \ar[r]^-{\omega_V} \ar[d]^-{\Phi} & F_V \ar[r]^-{\rho_V} & \Ob_V \\
\Delta \ar[ur]_-g \ar[u] & W \ar[r]_-{\omega_W} & F_W \ar[u]_-{\eta_\Phi} \ar[r]_-{\rho_W} & \Ob_W, \ar[u]_-{\Phi^\ck}
}
\end{align*}
\eqref{6.30} follows. \end{proof}

We now show that $\Omega$-compatibility is preserved under taking Kirwan blowups and slices thereof. We begin with a preparatory lemma. 

\begin{lem} \label{6.21}
Let $\Lambda_V$, $\Lambda_W$ as above. Then the induced embedding of Kirwan blowups $\hat{\Phi} \colon \hV \to \hW$ induces a morphism $\hat{\Phi}^\ck \colon \Ob_{\hW} \to \Ob_{\hV}$. The same is true for an \'{e}tale slice of a closed point of $\hU$ with closed orbit. 
\end{lem}

\begin{proof}
We need to show that $\eta_{\hat{\Phi}}$ maps $\im (d_{\hW} \omega_{\hW}^\vee)^\vee|_{\hU}$ to $\im (d_{\hV} \omega_{\hV}^\vee)^\vee|_{\hU}$. Since $\eta_{\hat{\Phi}}$ maps $T_{\hV}|_{\hU}$ into $\im (d_{\hV} \omega_{\hV}^\vee)^\vee|_{\hU}$, it suffices to show that the same is true for $N_{\hV/\hW}|_{\hU}$, for any local splitting $T_{\hW} |_{\hV} = T_{\hV} \oplus N_{\hV / \hW}$. 

$\eta_\Phi$ satisfies the same requirement, so by the same reasoning we may find a morphism $\alpha \colon N_{V/W}|_U \to T_V|_U$ such that the following diagram commutes (for a splitting $T_W|_V = T_V \oplus N_{V/W}$)
\begin{align*}
\xymatrix{
N_{V/W}|_U \ar[r] \ar[dr]_-\alpha & T_W|_U \ar[r]^-{(d_{W} \omega_{W}^\vee)^\vee} & F_W|_U \ar[d]^-{\eta_\Phi} \\
& T_V |_U \ar[r]_-{(d_{V} \omega_{V}^\vee)^\vee} & F_V|_U.
}
\end{align*}
By Lemma~\ref{comparison} and Lemma~\ref{conormal transform}, we have that $\mathrm{bl}(T_V) \sub T_{\hV}$, $\mathrm{bl}(T_W|_V) \sub T_{\hW} |_{\hV}$ are subbundles and $\mathrm{bl}(N_{V/W}) = N_{\hV/\hW}$. We therefore obtain a commutative diagram
\begin{align*}
\xymatrix{
N_{\hV/\hW}|_U \ar[r] \ar[dr]_-{\hat{\alpha}} & \mathrm{bl}(T_W|_U) \ar[r] & T_{\hW}|_{\hU} \ar[r]^-{(d_{\hW} \omega_{\hW}^\vee)^\vee} & F_{\hW}|_{\hU} \ar[d]^-{\eta_{\hat{\Phi}}} \\
& \mathrm{bl}(T_V|_U) \ar[r] & T_{\hV} |_{\hU} \ar[r]_-{(d_{\hV} \omega_{\hV}^\vee)^\vee} & F_{\hV}|_{\hU},
}
\end{align*}
where the composition $N_{\hV/\hW}|_{\hU} \to \mathrm{bl}(T_W|_U) \to T_{\hW}|_{\hU}$ comes from the induced splitting $T_{\hW}|_{\hV} = T_{\hV} \oplus N_{\hV/\hW}$ and the desired conclusion follows.

Let now $S$ be an \'{e}tale slice for $\hat{W}$ of a closed point $x \in \hV \xrightarrow{\Phi} \hat{W}$ with stabilizer $H$ and $T = \hat{V} \cap S$, $R = \hat{U} \cap T = \hat{U} \cap S$. Then we have induced data $\Lambda_T$, $\Lambda_S$ of a local model on $T$ and $S$ (cf. Lemma~\ref{Obs bundle for blow}) respectively such that $R$ is in standard form.

By the definition of $\omega_S, \omega_T$ we see that the two horizontal compositions in the diagram
\begin{align*}
\xymatrix{
T_S |_R \ar[r] & T_{\hW} |_R \ar[r] & F_{\hW} |_R \ar[r] \ar[d]^-{\eta_\Phi} & F_S |_R \ar[d]^-{\eta_\Psi} \ar[r]^-{\rho_S} & \Ob_S \\
T_T |_R \ar[r] \ar[u] & T_{\hV} |_R \ar[r] \ar[u] & F_{\hV} |_R \ar[r] & F_T |_R \ar[r]^-{\rho_T} & \Ob_T
}
\end{align*}
equal $( d_S \omega_S^\vee )^\vee$ and $( d_T \omega_T^\vee)^\vee$. By Lemma~\ref{lemma complex} and an identical argument to Lemma~\ref{comparison for slice}, in order to show that $\eta_\Psi$ maps the image $\im ( d_S \omega_S^\vee )^\vee|_R$ into $\im ( d_T \omega_T^\vee)^\vee|_R$ we may replace all bundles by their reduced analogues and get a commutative diagram of sheaves
\begin{align*}
\xymatrix{
T_S |_R^\rred \ar[r] & T_{\hW} |_R^\rred \ar[r] & F_{\hW} |_R^\rred \ar[r] \ar[d]^-{\eta_\Phi} & F_S |_R^\rred \ar[d]^-{\eta_\Psi} \\
T_T |_R^\rred \ar[r] \ar[u] & T_{\hV} |_R^\rred \ar[r] \ar[u] & F_{\hV} |_R^\rred \ar[r] & F_T |_R^\rred
}
\end{align*}
where all horizontal arrows except those in the middle are isomorphisms. It follows immediately that we obtain an induced morphism $\Psi^\ck \colon \Ob_S \to \Ob_T$. \end{proof}

\begin{lem} \label{omega compat for blowups}
Let $\Lambda_V$, $\Lambda_W$ be as above. Then the induced sections $\omega_{\hV}$ and $\omega_{\hW}$ (cf. Lemma~\ref{Obs bundle for blow}) are $\Omega$-compatible.
\end{lem}

\begin{proof}
By Lemma~\ref{comparison}, it readily follows that conditions (1) and (2) in the definition of $\Omega$-compatibility hold for the respective Kirwan blowups.

It remains to check condition (3). By Lemma~\ref{6.21}, we see that $\eta_{\hat{\Phi}}$ induces a morphism $\hat{\Phi}^\ck \colon \Ob_{\hW} \to \Ob_{\hV}$. Moreover, since $\eta_{\hat{\Phi}}$ is surjective, $\hat{\Phi}^\ck$ is surjective. We need to check that it is an isomorphism.
\begin{align} \label{6.26}
\xymatrix{
T_{\hW}|_{\hU} \ar[r]^-{(d_{\hW} \omega_{\hW}^\vee)^\vee} & F_{\hW}|_{\hU} \ar[d]^-{\eta_{\hat{\Phi}}} \ar[r]^-{\rho_W} & \Ob_{\hW} \ar[d]^-{\hat{\Phi}^\ck} \ar[r] & 0\\
T_{\hV} |_{\hU} \ar[r]_-{(d_{\hV} \omega_{\hV}^\vee)^\vee} & F_{\hV}|_{\hU} \ar[r]^-{\rho_V} & \Ob_{\hV} \ar[r] & 0.
}
\end{align}
Consider a morphism $\Delta \to \hU$, where $\Delta$ is the spectrum of a local Artinian ring. We will show that $\ell(\Ob_{\hW} |_\Delta) = \ell(\Ob_{\hV} |_\Delta)$. Since length is additive in exact sequences, we have 
\begin{align} \label{8.7}
\ell(\Ob_{\hW}|_\Delta) = \ell(F_{\hW}|_\Delta) - \ell (\ker \rho_W|_\Delta)
\end{align}
We have an exact sequence 
$$ T_{\hW} |_\Delta \xrightarrow{(d_{\hW}\omega_{\hW}^\vee)^\vee |_\Delta} F_{\hW} |_\Delta \xrightarrow{\rho_W |_\Delta} Ob_{\hW} |_\Delta \lra 0.$$
Thus $\ker \rho_W|_\Delta = \im (d_{\hW}\omega_{\hW}^\vee)^\vee |_\Delta$ and it follows that 
\begin{align} \label{8.8}
\ell (\ker \rho_W|_\Delta) = \ell( \im (d_{\hW}\omega_{\hW}^\vee)^\vee |_\Delta ) = \ell( T_{\hW} |_\Delta ) - \ell ( \ker (d_{\hW}\omega_{\hW}^\vee)^\vee |_\Delta )
\end{align}
Notice now that 
\begin{align} \label{8.9}
\ker (d_{\hW}\omega_{\hW}^\vee)^\vee |_\Delta = \coker \left( d_{\hW}\omega_{\hW}^\vee |_\Delta \right)^\vee = \Omega_{\hU}|_\Delta^\vee
\end{align} 
Therefore combining \eqref{8.7}, \eqref{8.8} and \eqref{8.9} we get, since $F_{\hW}$ and $T_{\hW}$ have the same rank, $\ell(\Ob_{\hW}|_\Delta) = \ell (\Omega_{\hU}|_\Delta^\vee)$.

An identical argument shows that $\ell(\Ob_{\hV}|_\Delta) = \ell (\Omega_{\hU}|_\Delta^\vee)$. 

Since $\ell(\Ob_{\hW}|_\Delta) = \ell(\Ob_{\hV}|_\Delta)$ for all such $\Delta \to \hU$ and $\hat{\Phi}^\ck$ is surjective, we conclude that $\hat{\Phi}^\ck$ is an isomorphism.\end{proof}

\begin{rmk}
It might seem counterintuitive that there is an induced map $\hat{\Phi}^\ck$, given that in the diagram \eqref{6.26} the derivative arrow $T_{\hV} |_{\hU} \to T_{\hW} |_{\hU}$ goes ``the wrong way". However, the fact that we begin with d-critical charts and then perform Kirwan blowups allows us to show that this is possible. The corresponding situation for taking \'{e}tale slices also owes to the special properties of Kirwan blowups.

Moreover, the isomorphism property of $\hat{\Phi}^\ck$ depends crucially on the fact that the 2-term complexes in question induce perfect obstruction theories on $\hU$. This enables us to show that the kernels of $(d_{\hW}\omega_{\hW}^\vee)^\vee |_\Delta$ and $(d_{\hV}\omega_{\hV}^\vee)^\vee |_\Delta$ are the same and hence obtain an equality of lengths.
\end{rmk}

\begin{lem} \label{omega compat for slices}
Let $\Lambda_V$, $\Lambda_W$ be as above. Let $S$ be an \'{e}tale slice for $\hat{W}$ of a closed point $x \in \hV \xrightarrow{\Phi} \hat{W}$ with stabilizer $H$ and $T = \hat{V} \cap S$, $R = \hat{U} \cap T = \hat{U} \cap S$. Then we have induced data $\Lambda_T$, $\Lambda_S$ of a local model on $T$ and $S$ (cf. Lemma~\ref{Obs bundle for blow}) respectively such that $R$ is in standard form and $\omega_T$ and $\omega_S$ are $\Omega$-compatible.
\end{lem}

\begin{proof}
By Lemma~\ref{comparison for slice}, $\eta_\Psi (\omega_S |_T)$ and $\omega_T$ are $\Omega$-equivalent, since $\eta_\Phi(\omega_{\hW}|_{\hV})$ and $\omega_{\hV}$ are. Thus, by Lemma~\ref{6.21}, conditions (1) and (2) of the definition of $\Omega$-compatibility hold. Condition (3) follows by an identical argument involving lengths as in the previous lemma.
\end{proof}

The induced map $\Phi^\ck$ is independent of the particular choice of embedding $\Phi$.

\begin{lem} \label{cocyle for blow and slice}
Let $U \sub V \xrightarrow{\Phi, \Psi} W$ be $G$-equivariant embeddings of affine $G$-schemes such that $U$ is in standard form for data $\Lambda_V$ and $\Lambda_W$, $\omega_V$ and $\eta_\Phi(\omega_W|_V)$ are $\Omega$-equivalent, $\omega_V$ and $\eta_\Psi(\omega_W|_V)$ are $\Omega$-equivalent and the diagram 
\begin{align*}
\xymatrix{
U \ar[r] \ar[dr] & V \ar[d]^-{\Phi, \Psi} \\
 & W
}
\end{align*}
commutes. If $\Phi^\ck = \Psi^\ck$, then the induced morphisms on the Kirwan blowups satisfy $\hat{\Phi}^\ck = \hat{\Psi}^\ck$. The same holds if we take slices of these blowups.

Furthermore, if $\Lambda_V$ and $\Lambda_W$ are data of d-critical charts for $U$, then we indeed have $\Phi^\ck = \Psi^\ck$ for any two choices of embedding $V \to W$.
\end{lem}

\begin{proof}
It suffices to show that if $\eta_\Phi - \eta_\Psi$ maps $F_W|_U$ to $\im (d_V \omega_V^\vee)^\vee|_U$ then the same is true for $\hat{\Phi} - \hat{\Psi}$. Let us denote $\alpha = \eta_\Phi - \eta_\Psi$ for brevity. 

Since $\alpha|_U$ maps $F_W|_U$ to $\im (d_V \omega_V^\vee)^\vee|_U$, it follows that $\alpha$ factors as $$F_W |_U \xrightarrow{\beta} T_V |_U \xrightarrow{(d_V \omega_V^\vee)^\vee} F_V|_U.$$ 
Similarly to Lemma~\ref{6.21}, we obtain a commutative diagram
\begin{align*}
\xymatrix{
F_{\hW} |_{\hU} \ar[r]^-{\hat{\alpha}} \ar[dr]_-{\hat{\beta}} & F_{\hV}|_{\hU} \ar[r]^-{\gamma} & \pi^* F_V |_{\hU} \\
& T_{\hV}|_{\hU} \ar[r]_-{d\pi} & \pi^* T_V |_{\hU}. \ar[u]_{\pi^* (d_V \omega_V^\vee)^\vee}
}
\end{align*}

Hence $\gamma \circ \hat{\alpha} = \pi^* (d_V \omega_V^\vee)^\vee \circ d\pi \circ \hat{\beta} = \gamma \circ (d_{\hV} \omega_{\hV}^\vee)^\vee \circ \hat{\beta}$ on $\hU$, from which we deduce that $\gamma \circ \left( \hat{\alpha} - (d_{\hV} \omega_{\hV}^\vee)^\vee \circ \hat{\beta} \right) \in \pi^* I_U \cdot \pi^* F_V$, which in turn yields that $\hat{\alpha} - (d_{\hV} \omega_{\hV}^\vee)^\vee \circ \hat{\beta} \in I_{\hU} \cdot F_{\hV}$ and therefore $\hat{\alpha}|_{\hU} = (d_{\hV} \omega_{\hV}^\vee)^\vee|_{\hU} \circ \hat{\beta}|_{\hU}$, which implies what we want.

The proof for slices proceeds along the same lines of Lemma~\ref{comparison for slice} and Lemma~\ref{6.21} and we omit it. 

Finally if the two sets of data give d-critical charts of $U$, we have $\Phi^* - \Psi^* \colon \oO_W \to \im \omega_V^\vee \sub \oO_V$ for the map on coordinate rings and therefore by symmetry $\eta_\Phi - \eta_\Psi \colon \Omega_W |_U \to \im \left( d_V \omega_V^\vee \right)|_U = \im \left( d_V \omega_V^\vee \right)^\vee|_U \sub \Omega_V|_U$, which implies the desired equality $\Phi^\ck = \Psi^\ck$.
\end{proof} 

The following lemma asserts that taking a slice gives compatible reduced obstruction sheaves and assignments and concludes this section.

\begin{lem} \label{passing to a slice red}
Let $(U, V,F_V,\omega_V,D_V,\phi_V)$ be the data of a local model on $V$. Let $\Phi \colon S \to V$ be an \'{e}tale slice of a closed point of $V$ with closed $G$-orbit and stabilizer $H$ and $(T, S, F_S, \omega_S, D_S, \phi_S)$ be the induced data on $S$.

Consider the diagram
\begin{align} \label{7.2}
\xymatrix{
K_V|_T \ar@{=}[r] & [ \fg \ar[r] & T_V |_T \ar[r]^-{ ( d_V \omega_V^\vee )^\vee } & F_V |_T \ar[r] \ar[d] & \fg^\vee (-D_S) ] \ar[d] \\
K_S \ar@{=}[r] & [ \fh \ar[u]  \ar[r] & T_S |_T \ar[r]_-{ ( d_S \omega_S^\vee )^\vee } \ar[u] & F_S |_T \ar[r] & \fh^\vee (-D_S) ]
}
\end{align}
The surjection $\eta_\Phi \colon F_V |_S \to F_S$ induces an isomorphism $\Phi^b \colon H^1(K_V|_T) \to H^1(K_S)$ which restricts to an isomorphism $\Ob_V^\rred|_{T^s} \to \Ob_S^\rred$, via which the dual complexes $\left( K_V^\rred |_T \right)^\vee, \left( K_S^\rred \right)^\vee$ give the same obstruction assignments on $T$.
\end{lem}

\begin{proof}
We have a commutative diagram of short exact sequences
\begin{align*}
\xymatrix{
0 \ar[r] & \fh \ar[r] \ar[d] & \fg \ar[r] \ar[d] & \fg / \fh \ar@{=}[d] \ar[r] & 0 \\
0 \ar[r] & T_S|_T \ar[r] & T_V|_T \ar[r]  & \fg / \fh \ar[r] & 0
}
\end{align*}
inducing an isomorphism $T_S|_T^\rred \to T_V|_T^\rred$.

By the definition of $F_S$, we have a diagram of short exact sequences
\begin{align*}
\xymatrix{
0 \ar[r] & (\fg / \fh)^\vee (-D_T) \ar[r] \ar@{=}[d] & F_V |_T \ar[r] \ar[d] & F_S|_T \ar[r] \ar[d] & 0 \\
0 \ar[r] & (\fg / \fh)^\vee (-D_T) \ar[r] & \fg^\vee(-D_T) \ar[r] & \fh^\vee(-D_T) \ar[r] & 0
}
\end{align*}
and similarly we deduce that $F_V|_T^\rred \to F_S|_T^\rred$ is an isomorphism. Then the central square of \eqref{7.2} factors through the commutative diagram
\begin{align*}
\xymatrix{
T_V |_T^\rred \ar[r] & F_V|_T^\rred \ar[d] \\
T_S |_T^\rred \ar[u] \ar[r] & F_S |_T^\rred,
}
\end{align*}
where the vertical arrows are isomorphisms. We obtain an induced isomorphism $H^1(K_V|_T) \to H^1(K_S)$ restricting to an isomorphism $\Ob_V^\rred|_{T^s} \to \Ob_S^\rred$ by \eqref{H^1 to Ob} and the obstruction assignments of the reduced complexes must match by standard arguments as in the proof of Proposition~\ref{obs assign for omega compat}.
\end{proof} 
\begin{rmk} \label{emb to unram}
All of the above results on $\Omega$-compatibility are true if one replaces locally closed embeddings $\Phi \colon V \to W$ by unramified morphisms and Zariski open embeddings by \'{e}tale maps. This is because for our purposes it suffices to work \'{e}tale locally and then Zariski open maps correspond to \'{e}tale maps and locally closed embeddings to unramified morphisms. We will tacitly use this observation in the following sections of the paper.

One may alternatively choose to work in the complex analytic topology where everything works verbatim.
\end{rmk}

\section{Generalized Donaldson-Thomas Invariants} \label{gendt}

This section leads up to the main result of the paper, the construction of the generalized DT invariant via 
Kirwan blowups, by combining the results of the previous sections on the intrinsic stabilizer reduction of a quotient stack, d-critical loci, obstruction theory and the local computations of Section~\ref{LoCalc}.

\subsection{Obstruction theory of intrinsic stabilizer reductions of equivariant d-critical loci}

As in Section~\ref{Kirwan blow-up}, let $\mM = \left[ X / G \right]$ be a quotient stack obtained by GIT. We have therefore a $G$-equivariant embedding $X \sub P$ where $P$ is a smooth scheme with a linearized $G$-action.

Suppose also that $\mM$ admits a d-critical structure $s \in \Gamma(\mM, \sS_\mM)$ so that $(\mM,s)$ is a d-critical stack. By Remark~\ref{4.11} this is equivalent to a $G$-invariant d-critical structure on $X$.

In this section, we carefully follow the steps of Section~\ref{Kirwan blow-up} tailored to the case of a d-critical locus to show that the intrinsic stabilizer reduction $\tilde{\mM}$ admits a semi-perfect obstruction theory.

Let $\fR(X) = \lbrace R_1, ..., R_m, \lbrace 1 \rbrace \rbrace$ in order of decreasing dimension.

Let $x \in X$ be a closed point with closed $G$-orbit, fixed by $R_1$. Since the action of $G$ is good, we have a $G$-invariant affine open $x \in U \sub X$. Therefore, we may apply the \'etale slice theorem \cite[Theorem~5.3]{Drezet} to get a locally closed affine $x \in T \sub U$ such that $$[T/R_1] \to [U/G] \sub [X/G]$$ is \'etale. 

Considering all such points $x \in X$, we obtain an \'etale cover
$$ \coprod [T\lalp^1 / R_1] \coprod [ (X - GZ_{R_1})/G ] \lr [X/G] = \mM,$$
where each $T\lalp^1$ is an $R_1$-invariant d-critical locus. 

In particular, for each index $\alpha$ there exist data of a local model $$(T\lalp^1, S\lalp^1, \Omega_{S\lalp^1}^1, df\lalp,0,\sigma_{S\lalp^1})$$ with $T\lalp^1 = (df\lalp = 0) \sub S\lalp^1$ in standard form and we may take $S\lalp^1$ to be affine, locally closed in $P$.

Let $X_1 = \hat{X}, P_1 = \hat{P}$ be the Kirwan blowups of $X_1, P_1$ respectively associated with $R_1$ and set $X_1^\circ = X - GZ_{R_1}$. Then we obtain an induced \'etale cover
$$ \coprod [\hat{T}\lalp^1 / R_1] \coprod [ X_1^\circ/G ] \lr [X_1/G] = \mM_1 \sub \pP_1 := [P_1/G].$$
For each $\alpha$, we have induced data $$(\hat{T}\lalp^1, \hat{S}\lalp^1, F_{\hat{S}\lalp^1}, \omega_{\hat{S}\lalp^1}, D_{\hat{S}\lalp^1}, \phi_{\hat{S}\lalp^1})$$ with $\hat{T}\lalp^1 = (\omega_{\hat{S}\lalp^1} = 0) \sub \hat{S}\lalp^1 \sub \hat{P} = P_1$. Note also that $[ X_1^\circ /G ] \lr \mM_1$ factors through $\mM$ as an open immersion.

Now $R_2 \in \fR(X_1)$ is of maximal dimension. Let $x \in X_1$ be a closed point with closed $G$-orbit, fixed by $R_2$. Then its orbit will lie in $X_1^\circ$ or it will be contained in the image of some $[\hat{T}\lalp^1 / R_1]$. Applying the same reasoning, we get an induced \'etale cover
$$ \coprod [\hat{T}\lbet^2 / R_2 ] \coprod [(\hat{T}\lalp^1)^\circ / R_1] \coprod [ X_2^\circ /G ] \lr [X_2/G] = \mM_2 \sub \pP_2 := [P_2/G],$$
where $(\hat{T}\lalp^1)^\circ = \hat{T}\lalp^1 - \hat{T}\lalp^1 \cap GZ_{R_2},\ X_2^\circ = X_1^\circ - GZ_{R_2}$ and the various ${T}\lbet^2 \sub S\lbet^2$ are \'etale slices of $G \times_{R_1} \hat{T}\lalp^1$ or $X_1^\circ$ in standard form, where we may take $S\lbet^2$ to be slices in $P_2 = \hat{P_1}$.

Continuing inductively, we have for any $n$ an \'{e}tale cover
$$\coprod [\hat{T}_\gamma^{n} / R_{n}] \coprod \dots \coprod [(\hat{T}\lalp^{1})^\circ / R_{1}] \coprod [ X_{n}^\circ /G] \to \mM_n \sub \pP_n := [P_n /G],$$
where (by abuse of notation) we write $$(\hat{T}\lalp^{i})^\circ = \hat{T}\lalp^{i} - GZ_{R_{i+1}} \cap  \hat{T}\lalp^{i} - ... - GZ_{R_{n}} \cap  \hat{T}\lalp^{i}$$ and so on, and the $T\lalp^{n}$ are appropriate slices of the elements of the \'etale cover for $\mM_{n-1}$. Note also that for each $i$, $[(\hat{T}\lalp^{i})^\circ / R_{i}]$ factors through an \'{e}tale morphism to $\mM_i$.

We see that $\tilde{\mM} = \mM_m = [ X_m / G ] \sub [P_m/G]$, a DM stack, is the intrinsic stabilizer reduction of $\mM$. We formalize the above procedure in the following lemma, where we also keep track of obstruction sheaves and their comparison data.

\begin{lem} \label{long prop}
For each $n \geq 0$, let $\eE_n^X, \eE_n^P$ be the union of all exceptional divisors in $\mM_n$ and $\pP_n$ respectively. There exist collections of \'{e}tale morphisms $[T\lalp / R\lalp] \to \mM_n$ and $[S\lalp / R\lalp] \to \pP_n$ such that:
\begin{enumerate}
\item For each $\alpha$, $R\lalp \in \lbrace R_1, ..., R_n \rbrace$.
\item Each $T\lalp$ is in standard form for data $(T\lalp, S\lalp, F_{S\lalp}, \omega_{S\lalp}, D_{S\lalp}, \phi_{S\lalp})$ of a local model on a smooth affine $R\lalp$-scheme $S\lalp \sub P_n$.
\item The collections cover $\eE_n^X$ and $\eE_n^P$ respectively.
\item The identity components of stabilizers that occur in $\mM_n$ lie, up to conjugacy, in the set $\lbrace R_{n+1}, ..., R_m \rbrace$.
\item The data $(T\lalp, S\lalp, F_{S\lalp}, \omega_{S\lalp}, D_{S\lalp}, \phi_{S\lalp})$ restricted to the complement of $\eE_n^P \sub \pP_n$ are the same as those of a d-critical chart on $T\lalp$.
\item For $\alpha, \beta$ and $q \in [T\lalp / R\lalp] \times_{\mM_n} [T\lbet / R\lbet]$ whose stabilizer has identity component conjugate to $R_q$, there exist an affine $T\lab$ in standard form for data $(T\lab, S\lab, F\lab, \omega\lab, D\lab, \phi\lab)$ of a local model and an equivariant commutative diagram
\begin{align} \label{rhombus}
\xymatrix{
& T\lab \ar[r] \ar[dl]_-{i\lalp} & S\lab \ar[dr]_-{\theta\lbet} \ar[dl]^-{\theta\lalp} & \ar[l] T\lab \ar[dr]^-{i\lbet} &\\ 
T\lalp \ar[r] \ar[dr] & S\lalp \ar[dr] & & S\lbet \ar[dl] & \ar[l] \ar[dl] T\lbet \\
& \mM_n \ar[r] & \pP_n & \ar[l] \mM_n, &
}
\end{align}
inducing a commutative diagram consisting of \'{e}tale maps on the corresponding quotient stacks for arrows pointing downwards. $T\lab, S\lab$ are \'{e}tale slices for both $T\lalp, T\lbet$ and $S\lalp, S\lbet$ respectively. All horizontal arrows are embeddings and $\theta\lalp, \theta\lbet$ are unramified and $R_q$-equivariant.
$\eta_{\theta\lalp} (\omega_{S\lalp})$ and $\eta_{\theta\lbet} (\omega_{S\lbet})$ are $\Omega$-equivalent to $\omega\lab$.

\item For each index $\alpha$, consider the 4-term complex
$$ K\lalp = [ \fr\lalp \lr T_{S\lalp} |_{T\lalp} \xrightarrow{(d\omega_{S\lalp}^\vee)^\vee} F_{S\lalp}|_{T\lalp} \xrightarrow{\phi_{S\lalp}} \fr\lalp^\vee(-D_{S\lalp})],$$
where by convention we place $F_{S\lalp}|_{T\lalp}$ in degree $1$. $\theta\lalp$ induces an isomorphism $\theta\lalp^b \colon \Ob_{S\lalp}^\rred |_{T\lab^s} \to \Ob_{S\lab}^\rred$. 
This does not change if we replace $\omega_{S\lalp}$ by any $\Omega$-equivalent section. The corresponding statements are true for the index $\beta$.

\item We obtain comparison isomorphisms $$\theta\lab^b := (\theta\lbet^b)^{-1} \theta\lalp^b \colon \Ob_{S\lalp}^\rred |_{T\lab^s} \to \Ob_{S\lbet}^\rred |_{T\lab^s}.$$ These give the same obstruction assignments for the complexes $K\lalp^\rred$, $K\lbet^\rred$ on $T\lab^s$.

\item Let now $q$ be a point in the triple intersection $$q \in [T\lalp / R\lalp] \times_{\mM_n} [T\lbet / R\lbet] \times_{\mM_n} [T_\gamma / R_\gamma]$$ with stabilizer in class $R_q$.  Then we have $S_{\lab}$ as in (6) for the indices $\alpha, \beta$, $S_{\beta\gamma}$ as in (6) for the indices $\beta, \gamma$ and $S_{\gamma \alpha}$ for the indices $\gamma, \alpha$ with a common $R_q$-invariant \'{e}tale neighbourhood $S_{\alpha \beta \gamma}$ fitting on top of the diagrams of the form \eqref{rhombus}.
The descents of $\theta\lab^b, \theta_{\beta \gamma}^b, \theta_{\gamma \alpha}^b$ to $[T_{\alpha \beta \gamma}^s/R_q]$ satisfy the cocycle condition.
\end{enumerate}
\end{lem}

\begin{proof} 
For $n = 0$ there is nothing to show, as we may take an empty set of \'{e}tale morphisms. We proceed by induction. Suppose the claim is true for $n$.

Consider the locus of closed points $x \in \mM_n$ whose stabilizer has identity component conjugate to $R_{n+1}$. Then, either $x$ is in the image of some $[T\lalp / R\lalp]$ or not.

Let's examine the first case. Take $[T\lalp / R\lalp]$ and consider all such closed points $x \in [T\lalp / R\lalp]$. For each $x$, let $[T\lalp^x / R_{n+1}] \to [T\lalp / R\lalp]$ be induced by an \'{e}tale slice $[S\lalp^x / R_{n+1}] \to [S\lalp / R\lalp]$. Then we consider the collection of \'{e}tale maps $[\hat{T\lalp^x} / R_{n+1}] \to \mM_{n+1}$ together with $[T\lalp^\circ / R\lalp] \to \mM_{n+1}$ where $T\lalp^\circ$ is the complement in $T\lalp$ of the locus of closed points with stabilizer whose identity component is conjugate to $R_{n+1}$.
We may repeat this process for all $[T\lalp / R\lalp] \to \mM_n$ and $x$. We say that these \'{e}tale maps are of type I.

So, consider $x \in \mM_n$ such that $x$ does not lie in the image of any $[T\lalp / R\lalp]$. In particular, by (3) $x$ does not lie in $\eE_n^X$ and thus we may assume that it is a closed point of $[X/G]$. Then there exist slices $T \sub S \sub P$ such that $T$ is in standard form for data $(T, S, \Omega_S, df_S, 0, \sigma_S)$ of a local model on $S$ and we have \'{e}tale morphisms $[T/R_{n+1}] \to \mM_n,\ [S/R_{n+1}] \to \pP_n$. Then, we may take \'{e}tale maps $[\hat{T} / R_{n+1}] \to \mM_{n+1}, \ [\hat{S} / R_{n+1}] \to \pP_{n+1}$. We may repeat this process for all such $x$. We say that these maps are of type II.

We have thus produced a collection of \'{e}tale morphisms for $\mM_{n+1}$ and $\pP_{n+1}$. It is clear by our choice and the inductive hypothesis that (1)-(5) are automatically satisfied by Lemma~\ref{Obs bundle for blow} and the properties of Kirwan blowups.

To check (6), the fact that the maps are unramified follows from the slice property since the derivatives are injective around the point of interest. Furthermore, since by (5) the restriction of a map of type I to the complement of $\eE_n^X$ clearly yields rise to a map of type II for $\mM_{n+1}$ it suffices to produce comparison data for maps of the same type. 

Now, for two maps of type I, we may assume that $S\lab$ in \eqref{rhombus} factors through $S\lalp^x$ and $S\lbet^y$ and therefore we may use $\hat{\theta}\lalp$ and $\hat{\theta}\lbet$ to get comparison data, which satisfy the requirements, since $\Omega$-equivalence is preserved by Kirwan blowups and taking slices by Lemma~\ref{comparison} and Lemma~\ref{comparison for slice}. 

For two maps of type II, coming from two choices of $[T/R_{n+1}] \to \mM$ and $[S/R_{n+1}] \to \pP$, we may find (up to shrinking), as in the proof of Proposition~\ref{indep} and using the properties of d-critical loci, a common \'etale refinement $T\lab \to T\lalp,\ T\lab \to T\lbet$ with 
$T\lab = (\omega\lab = df\lab = 0) \sub S\lab$ and commutative comparison diagrams
\begin{align} \label{comp diagram}
\xymatrix{
T\lab \ar[r] \ar[d] & S\lab \ar[d]^-{\theta\lalp} \\
T\lalp \ar[r] & S\lalp \ar[d]^-{f\lalp} \\
 & \bA^1
}
\end{align}
fitting in a diagram of the form \eqref{rhombus} for $\mM_0 := \mM$ and $\pP_0 := \pP$, such that $\theta\lalp, \theta\lbet$ are \'{e}tale and $d f\lalp|_{S\lab}$, $d f\lbet|_{S\lab}$ are $\Omega$-equivalent to $df\lab$. Since $\Omega$-equivalence is preserved for Kirwan blowups and taking slices, we see that (6) is satisfied in this case as well and exceptional divisors pull back to exceptional divisors (cf. Lemma~\ref{Obs bundle for blow}).

Finally (7) and (8) are an immediate consequence of the inductive hypothesis and Lemma~\ref{passing to a slice red}  of the preceding section.
The existence of $S_{\alpha \beta \gamma}$ in (9) can be seen, using the inductive hypothesis for maps of type I, and the d-critical structure for maps of type II. The cocycle condition follows from the commutativity of the diagrams of quotient stacks induced by \eqref{rhombus} and the fact that by construction $\theta\lalp^b$ is induced by pullback, is functorial with respect to compositions and both the reduced obstruction sheaves $\Ob_{S\lalp}^\rred$ and the morphisms $\theta\lalp$ descend to the level of quotient stacks. We note here that all the results of Section 6 hold true if one replaces embeddings by unramified morphisms, so they apply to the present situation as well (cf. Remark~\ref{emb to unram}). \end{proof}

We are now in position to show that $\tilde{\mM} = \mM_m$ is equipped with a semi-perfect obstruction theory of dimension $0$. Let $[T\lalp / R\lalp] \to \tilde{\mM}$ be the cover granted by the previous lemma. Each $[T\lalp / R\lalp]$ is a DM stack and together they cover the strictly semistable locus of $\mM$. Then for any $x \in \mM^s$, as in the proof of the lemma, we have an \'{e}tale map $T_x \to \mM$, where $T_x$ is a d-critical locus with $T_x \sub S_x \to P$. The stable locus $\mM^s$ is not affected during the Kirwan blowup procedure and so $\mM^s$ is an open substack of $\tilde{\mM}$ and we get \'{e}tale maps $T_x \to \tilde{\mM}$. We obtain an \'{e}tale surjective cover 
$$\coprod [ T\lalp / R\lalp ] \coprod T_x \lra \tilde{\mM}.$$
It is easy to check that the obstruction sheaves and assignments of the $T_x$ are compatible with those of the cover $[T\lalp / R\lalp] \to \tilde{\mM}$ by an identical argument as in the above proof. We thus see that the reduced obstruction theories indeed give a semi-perfect obstruction theory of dimension $0$ on $\tilde{\mM}$.

We have proved our main theorem.

\begin{thm} \label{main thm}
Let $\mM = \left[ X / G \right]$, where $X$ is a quasi-projective scheme, which is the semistable part of a projective scheme $X^\dagger$ with a linearized $G$-action.

Suppose also that $\mM$ admits a d-critical structure $s \in \Gamma(\mM, \sS_\mM)$ so that $(\mM,s)$ is a d-critical quotient stack. Then the intrinsic stabilizer reduction $\tilde{\mM}$ is a proper DM stack endowed with a canonical 
semi-perfect obstruction theory of dimension zero induced by $s$, giving rise to a 0-dimensional virtual fundamental cycle 
$[\tilde{\mM}]^{\mathrm{vir}} \in A_*(\tilde{\mM})$.
\end{thm}

\begin{rmk} \label{presentation independence}
Since at any point of our construction, we are working with stabilizers of closed points and slices thereof, $\tilde{\mM}$ and its obstruction theory are independent of the particular presentation of $\mM$ as a quotient stack $\left[ X / G \right]$.
\end{rmk}

\begin{rmk} \label{scale invariance}
If we replace the d-critical structure $s$ by $s' = c \cdot s$, where $c \in \bC^*$, the virtual cycle stays the same. This is because rescaling the d-critical structure is equivalent to replacing every d-critical chart $(U,V,f,i)$ of $X$ by $(U,V,cf,i)$ and this does not affect the intrinsic normal cone of the obstruction theory of $\tilde{\mM}$.
\end{rmk}

\subsection{Generalized DT invariants via Kirwan blowups} 
Let $\mM$ be the moduli space of Gieseker semistable sheaves with fixed Chern character $\gamma$ on a Calabi-Yau threefold $W$. Let $\kappa \colon \oO_W \to K_W$ be a trivialization of $K_W$. 

The space $\mM$ is the truncation of a derived Artin stack $\pmb{\mM}$. Moreover, by \cite{PTVV} $\kappa$ induces a 
$(-1)$-shifted symplectic structure on $\pmb{\mM}$ and therefore by \cite{JoyceArt} a d-critical structure on $\mM$. Since $\mM = \left[ X / G \right]$ is obtained by GIT (cf. \cite[Section~4]{HuyLehn}), the conditions of Theorem~\ref{main thm} hold and we have the intrinsic stabilizer reduction
$\tilde{\mM}$ and its virtual fundamental cycle $[\tilde{\mM}]^{\mathrm{vir}}$.

Since choosing a different trivialization amounts to rescaling $\kappa$ and subsequently the d-critical structure, by Remarks \ref{presentation independence} and \ref{scale invariance}, the virtual cycle does not depend on the choice of $\kappa$ or presentation as a quotient stack.

\begin{thm-defi} \emph{(Generalized DT invariant via Kirwan blowups)} \label{Gen DT}
We define the generalized DT invariant via Kirwan blowups of Chern character $\gamma$ to be 
$$\mathrm{DT}_\gamma(W) := \mathrm{deg} [\tilde{\mM}]^{\mathrm{vir}}.$$
\end{thm-defi}

\section{Relative Theory and Deformation Invariance} \label{definv}

In this section, we show that Kirwan blowups behave well over a smooth base curve $C$ and use this fact together with the relative versions of the results of the paper, obtained via derived symplectic geometry, to conclude that the generalized Donaldson-Thomas invariant via Kirwan blowups is invariant under deformations of the complex structure of the Calabi-Yau threefold $W$.

\subsection{Kirwan blowups in families}\label{Family blow} Kirwan blowups behave well in families over a smooth curve.

\begin{lem} \label{flatness of ideal of V^G}
Let $G$ be a reductive group; let $V$ be a smooth $G$-scheme, $C$ a smooth curve and $\pi \colon V \to C$ a smooth $G$-equivariant morphism, where $G$
acts on $C$ trivially. Then $V^G$ is smooth over $C$.
\end{lem}

\begin{proof} This is a standard fact as we are working over $\CC$.
First, as $V$ is smooth, $V^G$ is smooth. To prove that $V^G$ is smooth over $C$, we need to show that for
any closed $x\in V^G$, the projection $d\pi(x): T_x V^G\to T_{\pi(x)}C$ is surjective. 

Indeed, as both $V$ and $V^G$ are smooth, $T_x V^G=(T_x V)^G$. Let $v\in T_x V$ so that $d\pi(x)(v)\ne 0$. Applying the Reynolds operator
$\mathcal R$, we get that the $G$-invariant part of $v$, namely $\mathcal R(v)\in(T_x V)^G$, has $d\pi(x)(\mathcal R(v))
=d\pi(x)(v)$, thus $d\pi(x): T_x V^G\to T_{\pi(x)}C$ is surjective.
\end{proof}

The same proof gives the following result on \'etale slices.

\begin{lem} 
Let $\pi \colon V \to C$ be as in Lemma \ref{flatness of ideal of V^G}. If $x$ is a closed point in $V$ with reductive stabilizer $H$ and $S$ is an \'{e}tale slice for $x$, then $\pi \colon S \to C$ is smooth.
\end{lem}

\begin{cor} \label{intrinsic family V smooth}
Let $\pi \colon V \to C$ be as in Lemma \ref{flatness of ideal of V^G}. Then for any point $c \in C$ we have a canonical isomorphism $( {\hV} )_c \cong \hat{V_c}$.
\end{cor}
\begin{proof}
This follows immediately from the fact that $V^G$ is smooth over $C$.
\end{proof}

We obtain the following result on intrinsic stabilizer reductions.

\begin{prop} \label{intrinsic family U general}
Let $X=(X^\dagger)^{ss}$ with $X\sub P\sub \bP^N\times C$ be closed $G$-subschemes as before 
(cf. the beginning of Section~ \ref{Kirwan blow-up}) except where $C$ is a smooth curve and $G$ acts on $\bP^N$ via
a homomorphism $G\to GL(N+1)$. Let $\ti X$ be the intrinsic stabilizer reduction of $X$. Then for any closed point $c\in C$, $(\ti X)_c=\tilde{X_c}$.
\end{prop}
\begin{proof}
Let $\hat X$ be the Kirwan blowup of $X$ with respect to $G$. By the construction of the intrinsic stabilizer reduction, the lemma follows
from showing that $(\hat X)_c=\hat{X_c}$.
 
We are considering the case where $X$ comes with an equivariant $C$-embedding $X \sub V$ where $V \to C$ is smooth. Let $I \sub \oO_V$ be the ideal sheaf of $X \sub V$. Then, applying the Kirwan blowup $\pi \colon \hat{X} \to X$ we have a short exact sequence
$$0 \lra I^{intr} \lra \oO_{\hV} \lra \oO_{\hat X} \lra 0.$$
Let $\bC_c$ be the residue field at $c \in C$. We then have
$$ I^{intr} \otimes_{\oO_C} \bC_c \lra \oO_{(\hat V)_c} \lra \oO_{(\hat X)_c} \lra 0.$$ 
This fits into a diagram of exact sequences
\begin{align*}
\xymatrix{
I^{intr} \otimes_{\oO_C} \bC_c \ar[r] \ar[d] & \oO_{(\hat V)_c} \ar[r] \ar[d] & \oO_{(\hat X)_c} \ar[r] \ar[d] & 0 \\
(I \otimes_{\oO_C} \bC_c)^{intr} \ar[r] & \oO_{\hat{V_c}} \ar[r] &  \oO_{\hat{X_c}} \ar[r] & 0.
}
\end{align*}
By the preceding corollary, the middle arrow is an isomorphism. Moreover, if $I = I^{fix} \oplus I^{mv}$ is the decomposition of $I$ into its fixed and moving parts, since $G$ is reductive, $$I \otimes_{\oO_C} \bC_c = \left( I^{fix} \otimes_{\oO_C} \bC_c \right) \oplus \left( I^{mv} \otimes_{\oO_C} \bC_c \right)$$ is the decomposition into fixed and moving parts since the action of $G$ on $C$ is trivial. We conclude that the leftmost horizontal arrows have the same image under the identification $\oO_{(\hat V)_c} \simeq \oO_{\hat{V_c}}$ and thus we have a natural isomorphism $(\hat X)_c \simeq \hat{X_c}$.
\end{proof}

\subsection{Derived symplectic geometry: local picture} We assume for simplicity that $C = \Spec S$ is a smooth affine curve over $\bC$.

In what follows, we consider commutative differential graded algebras $(A^\bullet, \delta)$ (cdga's) over $S$, which are negatively graded. This will be the case for all differentials and cotangent complexes as well. Also, whenever we refer to a reductive group $G$, we assume that it acts trivially on $C$.

Results and properties of derived affine schemes that are quoted in this subsection can be found in the exposition of \cite[Sections~2,3]{JoyceSch}.
\begin{defi} \emph{(Standard form cdga)}
We say that a cdga $(A^\bullet, \delta)$ is of standard form if $A^0$ is a smooth $S$-algebra and, as a graded algebra, it is freely generated over $A^0$ by finitely many generators in each negative degree (i.e. it is quasifree).
\end{defi}
Now, a cdga $(A^\bullet, \delta)$ gives rise to the affine derived scheme $\dspec A^\bullet$ with underlying classical truncation the affine scheme $\Spec A^\bullet = \Spec H^0 (A^\bullet)$. $\dspec A^\bullet$ and $\Spec H^0(\Abul)$ have the same underlying topological space. 

If $A^\bullet$ is of standard form, then its derived cotangent complex $\bL_{A^\bullet}$ is represented by the K\"{a}hler differentials together with the internal differential $(\Omega_{A^\bullet}, \delta)$. We also have the usual de Rham differential $d$ on $\Omega_{A^\bullet}$, so that we obtain a mixed complex. Moreover, $\bL_{A^\bullet} \otimes_{A^\bullet} H^0(A^\bullet)$ is a complex of free $H^0(A^\bullet)$-modules.
\begin{defi} \emph{(Minimality)}
Let $(A^\bullet, \delta)$ be of standard form and $x \in \dspec A^\bullet$. We say that $A^\bullet$ is minimal at $x$ if all the differentials in $\bL_{A^\bullet}|_x$ are zero. 
\end{defi}
If $G$ is a reductive group acting on $(A^\bullet, \delta)$ (and by our convention trivially on $C$), then we have analogous equivariant statements. We will consider minimality at $x$ only when $x$ is a fixed point of $G$.

We proceed to give a brief account of $(-1)$-shifted symplectic forms on $\dspec A^\bullet$, introduced in \cite{PTVV}.

\begin{defi} \emph{($(-1)$-shifted symplectic form)} We say that $\pmb{\omega} = (\omega_0, \omega_1, ...)$ is a $(-1)$-shifted symplectic form on $\dspec A^\bullet$ if $\omega_i \in \left( \wedge^{2+i} \Omega_{A^\bullet} \right)^{-1-i}$ are such that
\begin{enumerate}
\item $\omega_0$ gives a quasi-isomorphism $\bL_{A^\bullet} \to \bT_{A^\bullet}[1]$,
\item $\delta \omega_0 = 0$ and $d \omega_i + \delta \omega_{i+1} = 0$ for $i \geq 0$.
\end{enumerate}
We refer to condition (1) as the non-degeneracy property and condition (2) as the closedness property.

Two forms $\pmb{\omega}, \pmb{\omega'}$ are equivalent if there exist $\alpha_i \in \left( \wedge^{2+i} \Omega_{A^\bullet} \right)^{-2-i}$ such that $\omega_0 - \omega_0' = \delta \alpha_0$ and for all $i \geq 0$, $\omega_{i+1} - \omega_{i+1}' = d \alpha_i + \delta \alpha_{i+1}$.
\end{defi}

A $(-1)$-shifted symplectic form can be defined on a derived stack by smooth descent. Suppose now that the derived quotient stack $\left[ \dspec \Abul/G \right]$ has a $(-1)$-shifted symplectic form $\pmb{\omega}$, where $\Abul$ is in standard form and minimal at a fixed point $x$. Then $\omega_0$ induces a quasi-isomorphism $$\bL_{[\dspec \Abul/G]}|_{H^0(\Abul)} \to \bT_{[\dspec \Abul/G]}|_{H^0(\Abul)}[1].$$

This implies that $\bL_{\Abul}$ must have Tor-amplitude $[-2,0]$ and therefore $\Abul$ is freely generated over $A^0$ in degrees $-1$ and $-2$ by generators $y_j \in A^{-1}$ and $w_k \in A^{-2}$ respectively. We may write the above quasi-isomorphism as the following equivariant morphism of complexes
\begin{align} \label{8.1}
\xymatrix{
V^{-2} \ar[r] \ar[d] & V^{-1} \ar[r] \ar[d] & \Omega_{A^0} \ar[d] \ar[r] & \fg^\vee \ar[d]\\
\fg \ar[r] & T_{A^0} \ar[r] & (V^{-1})^\vee \ar[r] & (V^{-2})^\vee.
}
\end{align}
Since $\Abul$ is minimal at the fixed point $x$, we may localize $G$-invariantly around $x$ and assume that the vertical arrows are isomorphisms.
In particular, we obtain an isomorphism $V^{-1} \simeq T_{A^0}$. Let $x_i$ be a set of \'{e}tale coordinates for $A^0$. We may now choose generators $y_i \in A^{-1}$ such that $dy_i \in V^{-1}$ are the dual basis to $d x_i \in \Omega_{A^0}$ under this isomorphism. We may thus identify $A^{-1}$ and $T_{A^0}$ as $A^0$-modules. Therefore the differential $\delta \colon T_{A^0} \to A^0$ induces an invariant section $\omega$ of $\Omega_{A^0}$ whose zero locus is precisely $\Spec H^0(\Abul)$. 

\begin{defi} \label{cdga}
We say that a cdga $(\Abul, \delta)$ with a $G$-action, minimal at a fixed point $x$ of $G$, is special if it is freely generated over $A^0$ in degrees $-1$ and $-2$ by generators $y_i$ and $w_j$ respectively, together with an identification $A^{-1} = V^{-1}$ (where $V^{-1}$ is as in \eqref{8.1}) mapping $y_i$ to $dy_i$. 

We denote $U = \Spec H^0(\Abul)$ and $V = \Spec A^0$.
\end{defi}

Every finitely presented affine derived scheme is (up to Zariski shrinking) equivalent to $\dspec \Abul$ with $\Abul$ a standard form cdga and this is also true in the equivariant setting around a fixed point $x$ of $G$. For more details, we refer to \cite[Theorem~4.1]{JoyceSch} and its proof, which is valid in the equivariant setting as well. We deduce the following proposition.

\begin{prop} \label{prop 8.8}
Let $U$ be an affine $G$-scheme over $C$, which is the classical scheme associated to a derived affine $G$-scheme $\textbf{U}$ such that the stack $[ \textbf{U} / G ]$ is $(-1)$-shifted symplectic with form $\pmb{\omega}$. Moreover, let $x \in U$ be a fixed point of $G$. Then, up to equivariant Zariski shrinking, $\textbf{U}$ is equivalent to $\dspec A^\bullet$, where $(\Abul, \delta)$ is a special cdga, minimal at $x$.  

In particular, there exists a smooth affine $G$-scheme $V \to C$, a $G$-equivariant embedding $U \to V$ over $C$ minimal at $x$ and an induced invariant 1-form $\omega \in H^0(\Omega_{V/C})$ such that $U = Z ( \omega ) \sub V$ is the zero locus of $\omega$.
\end{prop}

We can use the above to understand the local structure of quotient stacks that arise as the truncation of $(-1)$-shifted symplectic derived stacks. The following proposition can be deduced by work of Halpern-Leistner \cite{HL}.

\begin{prop} \label{derived local model}
Let $\pmb{\mM} \to C$ be a $(-1)$-shifted symplectic derived stack whose truncation $\mM = [ X / G ] \to C$ is a quotient stack such that the action of $G$ on $X$ is good (and trivial on $C$). Let $x \in \mM$ be a closed point with reductive stabilizer $R$. Then there exists an \'{e}tale morphism $\pmb{f} \colon [\pmb{T} / R] \to \pmb{\mM}$, a point $t \in [\pmb{T} / R]$ fixed by $R$, where $\pmb{T}$ is equivalent to $\dspec \Abul$ with $(\Abul, \delta)$ a special cdga, minimal at $t$, mapping $t$ to $x$ and inducing the inclusion $R \sub G$ on stabilizer groups.

At the classical level, we get an \'{e}tale morphism $f \colon [T/R] \to \mM$. There exist a smooth affine $R$-scheme $S \to C$, a $G$-equivariant embedding $T \to S$ over $C$ minimal at $t$ and an induced invariant 1-form $\omega \in H^0(\Omega_{S/C})$ such that $T = Z ( \omega ) \sub S$ is the zero locus of $\omega$.
\end{prop}

\begin{proof} Let $x \in U$ be a $G$-invariant affine open in $X$ (such exists since the $G$-action on $X$ is good). Since $U$ is affine and $G$ is reductive, by \cite[Lemma~4.2.5]{HL}, there exists an affine derived $G$-scheme $\pmb{U} = \dspec B^\bullet$ such that we have a fiber diagram 
\begin{align*}
\xymatrix{
[U/G] \ar[r] \ar[d] & \mM \ar[d]\\ 
[ \pmb{U} / G ] \ar[r] & \pmb{\mM}.
}
\end{align*}
Let $V = \Spec B^0$. By Luna's \'{e}tale slice theorem, we may pick an affine slice $x \in S$ in $V$ and obtain an \'{e}tale map $[T/R] \to [U/G]$, where $T = U \cap S$, induced from the \'{e}tale map $[S / R] \to [V / G]$. Using $S \to V$, there exists a derived affine scheme $\pmb{T} = \dspec C^\bullet$ with an $R$-action and $C^0 = B^0$, whose classical truncation is $T$, fitting in a fiber diagram
\begin{align*}
\xymatrix{
[T/R] \ar[r] \ar[d] & [U/G] \ar[d]\\ 
[ \pmb{T} / R ] \ar[r] & [ \pmb{U} / G ].
}
\end{align*}
where the lower horizontal arrow is \'{e}tale. Therefore $[ \pmb{T} / R ]$ is $(-1)$-shifted symplectic and we may apply Proposition~\ref{prop 8.8} to deduce that $\pmb{T}$ is equivalent to $\dspec \Abul$, where $(\Abul, \delta)$ is special and minimal at $t$.
\end{proof}

In analogy with the absolute case, we give the following definition.

\begin{defi}
We say that the tuple $\Lambda_V = (U, V, F_V, \omega_V, D_V, \phi_V)$ gives a relative local model structure on $U$ if $V$ is a smooth $G$-equivariant scheme over $C$, in addition to the rest of the data satisfying Setup-Definition~\ref{ind hyp} and one of the following:
\begin{enumerate}
\item $F_V = \Omega_{V/C}$, $D_V = 0$ and $\phi_V \colon \Omega_{V/C} \to \fg^\vee$ is the dual of the $G$-action. In this case, we call $\Lambda_V$ a quasi-critical chart on $V$.
\item $\Lambda_V$ is obtained by a quasi-critical chart by a sequence of Kirwan blowups and taking \'{e}tale slices of closed points with closed orbit.
\end{enumerate}
We then say that the $C$-scheme $U$ is in relative standard form.
\end{defi}

\begin{rmk} \label{abs -1 symple}
One of the main results of \cite{JoyceSch} is that, in the absolute case, if $\Abul$ is a standard form $\bC$-cdga with a $(-1)$-shifted symplectic form $\pmb{\omega}$, then up to equivalence and possible shrinking we have $\pmb{\omega} = (\omega_0, 0, 0, \dots)$, where $\omega_0 = d\phi$ and $\delta \phi = d\Phi$ for $\phi \in (\Omega_\Abul)^{-1}$ and $\Phi \in A^0$. In particular, this implies that $\Spec H^0 (\Abul)$ is the critical locus of $\Phi$ inside $\Spec A^0$. This also works equivariantly as in the above.
\end{rmk} 

\subsection{Comparison of local presentations} We first examine how the 1-form $\omega$ changes if one moves $\pmb{\omega}$ within its equivalence class.

\begin{prop} \label{prop 8.9}
Let $(\Abul, \delta)$ be a special cdga, minimal at $x$. Suppose that $\pmb{\omega}, \pmb{\eta}$ are equivalent $(-1)$-shifted symplectic forms on $[ \dspec \Abul / G ]$. Then, up to equivariant shrinking of $V$ around $x$, these induce 1-forms $\omega, \eta \in H^0(\Omega_{V/C})$ which are $\Omega$-equivalent:
\begin{enumerate}
\item We have an equality of ideals in $A^0$, $(\omega) = (\eta) = I_U$.
\item There exist equivariant morphisms $B, C \colon \Omega_{V/C} \to T_{V/C}$ such that
\begin{align*}
\omega^\vee - \eta^\vee = \eta^\vee B^\vee d\eta^\vee \ \text{mod} \ I_U^2
\end{align*}
and
\begin{align*}
\eta^\vee - \omega^\vee = \omega^\vee C^\vee d\omega^\vee \ \text{mod} \ I_U^2.
\end{align*}
\end{enumerate}
\end{prop}

\begin{proof} Let $x_i$ be an \'{e}tale basis for $V$ over $C$.
As in the discussion preceding Definition~\ref{cdga}, we may write $\omega_0 = \sum_i dy_i^\omega dx_i$ and $\eta_0 = \sum_i dy_i^\eta dx_i$, where $y_i^\omega$ and $y_i^\eta$ are bases for $A^{-1}$, for the 0-part of the pullbacks of $\pmb{\omega}$ and $\pmb{\eta}$ to $\dspec \Abul$. Thus we have
\begin{align*}
y_i^\omega = \sum_k J_{ik}^\omega y_k, \ y_i^\eta = \sum_k J_{ik}^\eta y_k, \ J_{ik}^\omega, J_{ik}^\eta \in A^0.
\end{align*}
The induced 1-forms are then given by 
\begin{align} \label{loc 1}
\omega = \sum_i \delta y_i^\omega \ dx_i = \sum_{i,k} J_{ik}^\omega s_k dx_i, \\ 
\notag \eta = \sum_i \delta y_i^\eta \ dx_i =\sum_{i,k} J_{ik}^\eta s_k dx_i,
\end{align}
where we write $s_i = \delta y_i$ for convenience. 
Since $y_i, y_i^\omega, y_i^\eta$ are all bases for $A^{-1}$ it is clear that $I_U = (\delta y_i) = (\delta y_i^\omega) = (\delta y_i^\eta)$, which proves (1).

Let $\delta w_j = \sum_i W_{ji} y_i$, where $W_{ji} \in A^0$. Since $\delta^2 (w_j) = 0$ we have
\begin{align} \label{loc 2}
\sum_{i,j} W_{ji} s_i = 0.
\end{align}

Now $\pmb{\omega}, \pmb{\eta}$ are equivalent as symplectic forms so in particular we have $\omega_0 - \eta_0 = \delta \alpha_0$ for some $\alpha_0 \in \left( \wedge^2 \Omega_{\Abul} \right)^{-2}$. By degree considerations, we may write
\begin{align*}
\alpha_0 = \sum_{ij} E_{ij} dy_i dy_j + \sum_{ik} F_{ik} dw_k dx_i + \alpha_0', \ E_{ij}, F_{ik} \in A^0,
\end{align*}
where $\alpha_0'$ is a 2-form, whose every term is divisible by some of the $y_i$, and we may assume without loss of generality that $E_{ij}$ is symmetric. Thus $\delta \alpha_0' \in I_U \cdot \left( \wedge^2 \Omega_{\Abul}^1 \right)^{-1}$ and we have
\begin{align*}
\delta \alpha_0 &  = -2\sum_{i,j} E_{ij} ds_i dy_j - \sum_{i,k} F_{ik} d(\delta w_k) dx_i \ \left( \text{mod} \ I_U \right) \\
& = -2\sum_{i,j} E_{ij} ds_i dy_j - \sum_{i,j,k} F_{ik} W_{kj} dx_i dy_j \ \left( \text{mod} \ I_U \right) .
\end{align*}
We have also
\begin{align*}
\omega_0 - \eta_0 & = \sum_i dy_i^\omega dx_i - \sum_i dy_i^\eta dx_i \\
& = \sum_{i,k} dJ_{ik}^\omega y_k dx_i - \sum_{i,k} J_{ik}^\omega dy_k dx_i - \sum_{i,k} dJ_{ik}^\eta y_k dx_i + \sum_{i,k} J_{ik}^\eta dy_k dx_i.
\end{align*}
By comparing the coefficients of each $dy_j$, we obtain a relation
\begin{align*}
\sum_i J_{ij}^\omega dx_i - \sum_i J_{ij}^\eta dx_i = 2 \sum_i E_{ij} ds_i + \sum_{i,k} F_{ik} W_{kj} dx_i  \ \left( \text{mod} \ I_U \right)
\end{align*}
for each index $j$. Then, using \eqref{loc 1},
\begin{align} \label{loc 3}
\omega - \eta & = \sum_{i,k} (J_{ik}^\omega dx_i - J_{ik}^\eta dx_i) s_k \\
\notag & = 2 \sum_{i,k} E_{ik} ds_i \ s_k + \sum_{i,k,j} F_{ij}W_{jk}s_k dx_i \ \left( \text{mod} \ I_U^2 \right)\\
\notag & = 2 \sum_{i,k} E_{ik} ds_i \ s_k \ \left( \text{mod} \ I_U^2 \right),
\end{align}
where in the second expression the second term is zero using \eqref{loc 2}.

Note that in order to derive \eqref{loc 3}, we only used the fact that $y_i$ is a basis for $A^{-1}$. So we could repeat the exact same analysis and obtain a similar equation with $s_i$ replaced by $s_i^\omega = \delta y_i^\omega$ or $s_i^\eta = \delta y_i^\eta$ and $E_{ik}$ by $E_{ik}^\omega$ or $E_{ik}^\eta$ respectively. It is then easy to see that these exactly imply (2), with the coefficients of $B, C$ being determined by $E_{ik}^\omega, \ E_{ik}^\eta$ after averaging over $G$ to make the morphism equivariant ($\omega - \eta$ is already invariant).
\end{proof}

Suppose now that we have \'{e}tale morphisms $\pmb{f}\lalp \colon [\pmb{T}\lalp / R\lalp] \to \pmb{\mM}$ and $\pmb{f}\lbet \colon [\pmb{T}\lbet / R\lbet] \to \pmb{\mM}$ as in Proposition~\ref{derived local model}, where $\pmb{T}\lalp$ is equivalent to $\dspec \Abul$ and $\pmb{T}\lbet$ is equivalent to $\dspec B^\bullet$, with $\Abul, B^\bullet$ special equivariant cdgas. Let $z \in [\pmb{T}\lalp / R\lalp] \times_{\pmb{\mM}} [\pmb{T}\lbet /R\lbet]$ be a closed point with stabilizer $H$. 

Note that $\pmb{f}\lalp, \pmb{f}\lbet$ are also affine. Then, since $T\lalp, T\lbet \to \mM$ are affine, the diagonal of $\mM$ is affine, and a derived scheme is affine if and only if its truncation is affine,  we obtain that $\pmb{T}\lalp \times_{\pmb{\mM}} \pmb{T}\lbet$ is an affine derived scheme with an action of $R\lalp \times R\lbet$ such that $[ \pmb{T}\lalp \times_{\pmb{\mM}} \pmb{T} \lbet / R\lalp \times R\lbet ]$ is $(-1)$-shifted symplectic. Then, there exists a special cdga $C^\bullet$ with an $H$-action, minimal at $t\lab$, such that we have $H$-equivariant morphisms $\alpha \colon \Abul \to C^\bullet,\ \beta \colon B^\bullet \to C^\bullet$ and a commutative diagram of \'{e}tale arrows
\begin{align} \label{derived commutat}
\xymatrix{
[\dspec C^\bullet / H] \ar[d]_-{\pmb{\theta\lalp}} \ar[r]^{\pmb{\theta\lbet}} & [\dspec B^\bullet / R\lbet] \ar[d]^-{\pmb{f}\lbet} \\
[\dspec A^\bullet / R\lalp] \ar[r]_-{\pmb{f}\lalp} & \pmb{\mM}.
}
\end{align}
Moreover, the morphism $[ \dspec C^\bullet / H ] \to [ \pmb{T}\lalp \times_{\pmb{\mM}} \pmb{T}\lbet / R\lalp \times R\lbet ]$ is \'{e}tale and maps $t\lab$ to $z$.

We can recast the above data at the level of classical stacks and schemes. We have an \'{e}tale map $f\lalp \colon [T\lalp / R\lalp ] \to \mM$, where $T\lalp \sub S\lalp$ is the zero locus of an invariant section $\omega\lalp$ of $\Omega_{S\lalp/C}$ and $S\lalp \to C$ is smooth, $R\lalp$-equivariant. Similar data is obtained for the \'{e}tale map $[T\lbet/R\lbet] \to \mM$. The above diagram shows that we have the following comparison data: 
\begin{enumerate}
\item We have an affine, smooth $H$-scheme $S\lab \to C$ with an invariant section $\omega\lab$ of $\Omega_{S\lab / C}$ with zero locus $T\lab$, minimal at a point $t\lab$ fixed by $H$. Here $T\lab$ is the truncation of the derived scheme $\dspec C^\bullet$.
\item There exist $H$-equivariant unramified morphisms $\theta\lalp \colon S\lab \to S\lalp$ and $\theta \lbet \colon S\lab \to S\lbet$, inducing unramified morphisms $T\lab \to T\lalp$ and $\ T\lab \to T\lbet$.
\item $\eta_{\theta\lalp} (\omega\lalp),\ \eta_{\theta\lbet} (\omega\lbet)$ are $\Omega$-equivalent to $\omega\lab$.
\item We have a commutative diagram with \'{e}tale arrows
\begin{align} \label{diagram}
\xymatrix{
[T\lab / H] \ar[d]_-{\theta\lalp} \ar[r]^-{\theta\lbet} & [T\lbet / R\lbet] \ar[d]^-{f\lbet} \\
[T\lalp / R\lalp] \ar[r]_-{f\lalp} & \mM.
}
\end{align}
\end{enumerate}

\begin{defi} \emph{(Common roof)}
If the above four conditions hold, coming from a diagram \eqref{derived commutat}, we say that the quasi-critical chart $\Lambda_{S\lab}$ is a common roof for the quasi-critical charts $\Lambda_{S\lalp}$ and $\Lambda_{S\lbet}$. More generally, the same definition applies to any two relative local models which are not necessarily quasi-critical charts.
\end{defi}

\begin{rmk} \label{cocycle for derived}
Suppose that we have a common roof coming from a commutative diagram \eqref{derived commutat} where $\pmb{\mM} = [ \dspec D^\bullet / G ]$ and $R\lalp = R\lbet = H$. Moreover, assume that we have two compositions $g\lalp \colon D^\bullet \to A^\bullet \to C^\bullet$ and $g\lbet \colon D^\bullet \to B^\bullet \to C^\bullet$ such that $\dspec g\lalp$, $\dspec g\lbet$ become equivalent when composed with the quotient morphism $\dspec D^\bullet \to [ \dspec D^\bullet / G ]$ and induce the diagram \eqref{derived commutat}. If $g\lalp, \ g\lbet \colon D_0 \to C_0$ are the induced morphisms in degree $0$ and we denote $V = \Spec D^0$, $$g\lalp^b = \theta\lalp^b \circ f\lalp^b, \ g\lbet^b = \theta\lbet^b \circ f\lbet^b \colon \Ob_V^\rred |_{T\lab^s} \lra \Ob_{S\lab}^\rred$$ are obtained by the corresponding maps from \eqref{derived commutat} on cotangent complexes of quotient stacks, as in Lemma~\ref{passing to a slice red}. It follows that 
for the (non-commutative) diagram of quotient stacks
\begin{align*}
\xymatrix{
[\Spec C^0 / H] \ar[d]_-{\theta\lalp} \ar[r]^-{\theta\lbet} & [\Spec B^0  / H] \ar[d]^-{f\lbet} \\
[\Spec A^0  / H] \ar[r]_-{f\lalp} & [ \Spec D^0  / G ].
}
\end{align*}
which is compatible with \eqref{diagram}, there is a natural equivalence between the compositions $\theta\lalp^b \circ f\lalp^b$ and $\theta\lbet^b \circ f\lbet^b$, which is also compatible with the commutativity of \eqref{diagram}.

Another way to see this more concretely is when $G$ is a Zariski open subscheme of an affine space. This is the case in our application to Donaldson-Thomas invariants, since closed points correspond to polystable sheaves whose stabilizers are products of general linear groups. The fact that $g\lalp, g\lbet$ composed with $\dspec D^\bullet \to [ \dspec D^\bullet / G ]$ are equivalent implies then that there exists a morphism $\pmb{h} = \dspec h \colon \dspec C^\bullet \to G$ such that the composition 
$$\dspec g\lbet' \colon \dspec C^\bullet \xrightarrow{(\id, \pmb{h})} \dspec C^\bullet \times G \xrightarrow{(\dspec g\lbet, \id)} \dspec D^\bullet \times G \lra \dspec D^\bullet,$$
where the last arrow is given by the group action, gives a map of cdgas $D^\bullet \to C^\bullet$ which is homotopic to $g\lbet$. In particular, the induced morphisms $g\lbet \colon D^0 \to C^0$ and $g\lbet' \colon D^0 
\to C^0$ satisfy $g\lbet - g\lbet' \colon D^0 \to \im C^{-1} = \im \omega\lab^\vee$ and thus, as in Lemma~\ref{cocyle for blow and slice}, we get that $$g\lbet^b = (g\lbet')^b \colon \Ob_V^\rred |_{T\lab^s} \lra \Ob_{S\lab}^\rred$$
and the induced morphism $h \colon \Spec C^0 \to G$ gives the data for the equivalence mentioned above. 
\end{rmk}

\subsection{Obstruction theory in the relative case} As before, suppose that $\pmb{\mM} \to C$ is a $(-1)$-shifted symplectic derived stack, whose truncation $\mM = [ X / G ]$ is a quotient stack obtained by GIT, where $X$ is a $C$-scheme and $G$ acts trivially on $C$.

The Kirwan partial desingularization procedure goes through in exactly the same way as in the absolute case. We have the following modified version of Proposition~\ref{long prop} in this relative situation.

The main difference stems from the fact that we need to use local models, arising from special cdga's (cf. Definition~8.8), which satisfy a minimality property (cf. Definition 8.6). Thus we can no longer use slices $S\lalp$ coming from a global embedding $\mM \to \pP$ into a smooth quotient stack and we will need to use the notion of $\Omega$-compatibility to address the issue of extra coordinates. In particular, in the analogue of diagram \eqref{rhombus}, we will be missing the arrows $S\lalp \to \pP_n, \ S\lbet \to \pP_n$ and $S\lalp, S\lbet$ can have extra coordinates.

\begin{prop} \label{long prop2}
For each $n \geq 0$, let $\eE_n^X$ be the union of all exceptional divisors in $\mM_n$. There exist collections of \'{e}tale morphisms $[T\lalp / R\lalp] \to \mM_n$ such that:
\begin{enumerate}
\item For each $\alpha$, $R\lalp \in \lbrace R_1, ..., R_n \rbrace$.
\item Each $T\lalp$ is in relative standard form for data $$\Lambda_{S\lalp} = (T\lalp, S\lalp, F\lalp, \omega\lalp, D\lalp, \phi\lalp)$$ of a relative local model on a smooth affine $R\lalp$-scheme $S\lalp$.
\item The collections cover $\eE_n^X$ respectively.
\item The identity components of stabilizers that occur in $\mM_n$ lie, up to conjugacy, in the set $\lbrace R_{n+1}, ..., R_m \rbrace$.
\item $\Lambda_{S\lalp}$ restricted to the complement of the union of exceptional divisors $\eE_{S\lalp} \sub S\lalp$ are the same as those of a quasi-critical chart on $T\lalp$.
\item For indices $\alpha, \beta$ and $q \in [T\lalp / R\lalp] \times_{\mM_n} [T\lbet / R\lbet]$ whose stabilizer has identity component conjugate to $R_q$, there exist an affine $T\lab$ in relative standard form for data $\Lambda_{S\lab}=(T\lab, S\lab, F\lab, \omega\lab, D\lab, \phi\lab)$ of a relative local model and an equivariant commutative diagram
\begin{align} \label{rhombus2}
\xymatrix{
& S\lab \ar[dl]_-{\theta\lalp} & T\lab \ar[dr]_-{i\lbet} \ar[dl]^-{i\lalp} \ar[l] \ar[r] &  S\lab \ar[dr]^-{\theta\lbet} &\\ 
S\lalp & T\lalp \ar[l] \ar[dr] & & T\lbet \ar[r] \ar[dl] & S\lbet \\
& & \mM_n, & &
}
\end{align}
such that $\Lambda_{S\lab}$ is a common roof for $\Lambda_{S\lalp}$ and $\Lambda_{S\lbet}$ coming from a diagram of the form \eqref{derived commutat}. Moreover, $\theta \lalp$ factors as $S\lab \to S\lalp' \to S\lalp$ where $S\lalp'$ is an \'{e}tale slice for $S\lalp$ at the image of $q$, $T\lab \to T\lalp'$ is \'{e}tale and $S\lab \to S\lalp'$ is unramified. Analogous conditions hold for the index $\beta$.
\item For each index $\alpha$, consider the 4-term complex
$$ K\lalp = [ \fr\lalp \lr T_{S\lalp} |_{T\lalp} \xrightarrow{(d\omega_{S\lalp}^\vee)^\vee} F\lalp|_{T\lalp} \xrightarrow{\phi_{S\lalp}} \fr\lalp^\vee(-D\lalp)],$$
where by convention we place $F_{S\lalp}|_{T\lalp}$ in degree $1$. $\theta\lalp$ induces an isomorphism $\theta\lalp^b \colon \Ob_{S\lalp}^\rred |_{T\lab^s} \to \Ob_{S\lab}^\rred$. 
This also does not change if we replace $\omega_{S\lalp}$ by any $\Omega$-equivalent section. Analogous statements are true for the index $\beta$.

\item We obtain comparison isomorphisms $$\theta\lab^b := (\theta\lbet^b)^{-1} \theta\lalp^b \colon \Ob_{S\lalp}^\rred |_{T\lab^s} \to \Ob_{S\lbet}^\rred |_{T\lab^s}.$$ These give the same obstruction assignments for the complexes $K\lalp^\rred$, $K\lbet^\rred$ on $T\lab^s$.

\item Let now $q$ be a point in the triple intersection $$q \in [T\lalp / R\lalp] \times_{\mM_n} [T\lbet / R\lbet] \times_{\mM_n} [T_\gamma / R_\gamma]$$ with stabilizer in class $R_q$. We then have a commutative schematic $R_q$-equivariant diagram of common roofs
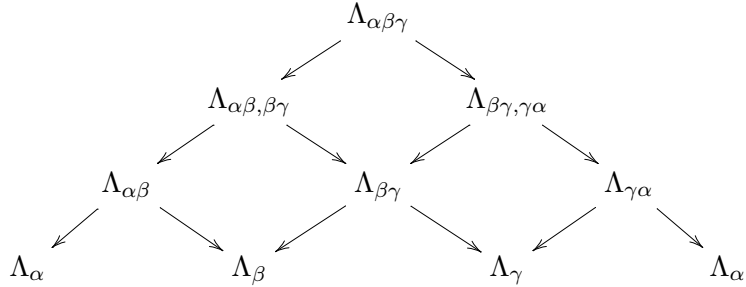
\begin{align*}
\xymatrixrowsep{0.2in}
\xymatrixcolsep{0.2in}
\xymatrix{
 & & & \Lambda_{\alpha \beta \gamma} \ar[dl] \ar[dr] & & & \\
& &  \Lambda_{\alpha\beta, \beta\gamma} \ar[dl] \ar[dr] & & \Lambda_{\beta\gamma, \gamma\alpha} \ar[dl] \ar[dr] \\
 & \Lambda\lab \ar[dl] \ar[dr]  & &  \Lambda_{\beta\gamma} \ar[dl] \ar[dr] & & \Lambda_{\gamma\alpha} \ar[dl] \ar[dr] \\
\Lambda\lalp & & \Lambda_\beta & & \Lambda_\gamma & & \Lambda\lalp
}
\end{align*}
such that all the morphisms between local models $\Lambda\lal \to \Lambda_\mu$, where $\lambda, \mu$ are multi-indices, satisfy the properties of a roof, the morphisms $T\lal \to T_\mu$ are \'{e}tale, and we have induced isomorphisms $\Ob_\mu^\rred |_{T\lal^s} \to \Ob\lal^\rred$.
The descents of $\theta\lab^b, \theta_{\beta \gamma}^b, \theta_{\gamma \alpha}^b$ to $[T_{\alpha \beta \gamma}^s / R_q]$ satisfy the cocycle condition.
\end{enumerate}
\end{prop}

\begin{proof} 
In the same way as in the absolute case, all conditions (1)-(9) are preserved at each inductive step. So we only need to check that they hold in the beginning and also address the differences in (6), (7) and (9) from the absolute case.

By the previous subsection, roofs exist and they satisfy the conditions of (6) by construction, since we may first replace $\pmb{T\lalp} / R\lalp$ by an \'{e}tale slice $\pmb{T\lalp}' / H$ (and similarly for the index $\beta$) and then take a common roof. In particular, by Lemma~\ref{omega compat for blowups} and Lemma~\ref{passing to a slice red} we immediately deduce that $\theta\lalp^b$ indeed induce isomorphisms on the reduced obstruction sheaves. It is at this point where we have to use the notion of $\Omega$-compatibility to deal with possible extra coordinates.

Finally, the cocycle condition holds by applying Remark~\ref{cocycle for derived} and using Lemma~\ref{omega compat for blowups}, ~\ref{omega compat for slices}, ~\ref{cocyle for blow and slice} and ~\ref{passing to a slice red}, which are valid in the relative setting, and going through the appropriate diagrams of roofs granted by (9). We leave the details to the reader. \end{proof}

We have shown that we can follow the same steps as in the absolute case to obtain the following theorem.

\begin{thm} \label{main thm rel}
Let $\mM = \left[ X / G \right] \to C$, where $X \to C$ is a quasi-projective scheme, which is the semistable part of a projective scheme $X^\dagger \to C$ with a linearized $G$-action and $C$ is a smooth, quasi-projective curve with a trivial $G$-action.
Suppose also that $\mM$ is the truncation of a relative $(-1)$-shifted symplectic derived stack $\pmb{\mM}$ over $C$. Then the intrinsic stabilizer reduction $\tilde{\mM} \to C$ is a proper DM stack over $C$ endowed with a canonical  relative semi-perfect obstruction theory induced by the relative $(-1)$-shifted symplectic form, giving rise to a virtual fundamental cycle 
$[\tilde{\mM}]^{\mathrm{vir}}$. 

Furthermore, the restriction of the obstruction theory to the fiber $\tilde{\mM}_c$ over a point $c \in C$ is the semi-perfect obstruction theory of Theorem~\ref{main thm}, constructed in the absolute case.
\end{thm}
\subsection{Deformation invariance of DTK invariants} Let $W \to C$ be a family of Calabi-Yau threefolds over a smooth affine curve. The relative canonical bundle $K_{W/C}$ is then trivial. 

Let $\mM$ be the moduli stack of relatively semistable sheaves over $C$ of Chern character $\gamma$ on $W$. As in the absolute case, a trivialization of $K_{W/C}$ induces a relative $(-1)$-shifted symplectic structure on the derived stack $\pmb{\mM} \to C$. Then by Theorem~\ref{main thm rel}, it follows that applying the same construction we obtain an intrinsic stabilizer reduction $\tilde{\mM} \to C$ with a relative semi-perfect obstruction theory, which induces the absolute semi-perfect obstruction theory on each fiber $\tilde{\mM}_c$ we constructed earlier. 

Moreover, its fiber over $c\in C$ is the intrinsic stabilizer reduction of the moduli stack of semistable sheaves on $W_c$ by the results of Subsection~\ref{Family blow}. 
Therefore, by ``conservation of number" for semi-perfect obstruction theories \cite[Proposition~3.8]{LiChang}, the generalized Donaldson-Thomas invariant stays the same along the family, i.e. $\mathrm{DT}_\gamma (W_c)$ is constant for $c \in C$.

\begin{thm} \label{deform invariance}
The generalized DT invariant defined in Theorem-Definition \ref{Gen DT} is invariant under deformation of the complex structure of the Calabi-Yau threefold.
\end{thm}

\begin{rmk}
Choosing a different trivialization of $K_{W/C}$ amounts to rescaling the induced trivializations of $K_{W_c}$ on each fiber, so by Remark~\ref{scale invariance} our discussion does not depend on the specific choice of trivialization.
\end{rmk}

\bibliography{Master}

\newcommand{\etalchar}[1]{$^{#1}$}
\begin{thebibliography}{MNOP06b}

\bibitem[AHH18]{AlpHalpHein}
Jarod {Alper}, Daniel {Halpern-Leistner}, and Jochen {Heinloth}.
\newblock {Existence of moduli spaces for algebraic stacks}.
\newblock {\em arXiv e-prints}, page arXiv:1812.01128, Dec 2018.

\bibitem[AHR19]{AHR2}
Jarod {Alper}, Jack {Hall}, and David {Rydh}.
\newblock {The {\'e}tale local structure of algebraic stacks}.
\newblock {\em arXiv e-prints}, page arXiv:1912.06162, December 2019.

\bibitem[AHR20]{Alper}
Jarod Alper, Jack Hall, and David Rydh.
\newblock A {L}una \'{e}tale slice theorem for algebraic stacks.
\newblock {\em Ann. of Math. (2)}, 191(3):675--738, 2020.

\bibitem[Alp13]{GoodAlper}
Jarod Alper.
\newblock Good moduli spaces for {A}rtin stacks [bons espaces de modules pour
  les champs d'{A}rtin].
\newblock {\em Annales de l'institut Fourier}, 63(6):2349--2402, 2013.

\bibitem[BBBBJ15]{JoyceArt}
Oren Ben-Bassat, Christopher Brav, Vittoria Bussi, and Dominic Joyce.
\newblock A `{D}arboux theorem' for shifted symplectic structures on derived
  {A}rtin stacks, with applications.
\newblock {\em Geom. Topol.}, 19(3):1287--1359, 2015.

\bibitem[BBJ19]{JoyceSch}
Christopher Brav, Vittoria Bussi, and Dominic Joyce.
\newblock A {D}arboux theorem for derived schemes with shifted symplectic
  structure.
\newblock {\em J. Amer. Math. Soc.}, 32(2):399--443, 2019.

\bibitem[Beh09]{BehFun}
Kai Behrend.
\newblock Donaldson-{T}homas type invariants via microlocal geometry.
\newblock {\em Ann. of Math. (2)}, 170(3):1307--1338, 2009.

\bibitem[BF97]{BehFan}
Kai Behrend and Barbara Fantechi.
\newblock The intrinsic normal cone.
\newblock {\em Invent. Math.}, 128(1):45--88, 1997.

\bibitem[BLM{\etalchar{+}}21]{familystab}
Arend Bayer, Mart\'{\i} Lahoz, Emanuele Macr\`\i, Howard Nuer, Alexander Perry,
  and Paolo Stellari.
\newblock Stability conditions in families.
\newblock {\em Publ. Math. Inst. Hautes \'{E}tudes Sci.}, 133:157--325, 2021.

\bibitem[BR16]{BehRon2}
Kai {Behrend} and Pooyah {Ronagh}.
\newblock {The Eigenvalue Spectrum of the Inertia Operator}.
\newblock {\em ArXiv e-prints}, December 2016.

\bibitem[BR19]{BehRon1}
Kai Behrend and Pooya Ronagh.
\newblock The inertia operator on the motivic {H}all algebra.
\newblock {\em Compos. Math.}, 155(3):528--598, 2019.

\bibitem[CL11]{LiChang}
Huai-{L}iang Chang and Jun Li.
\newblock Semi-perfect obstruction theory and {D}onaldson-{T}homas invariants
  of derived objects.
\newblock {\em Comm. Anal. Geom.}, 19(4):807--830, 2011.

\bibitem[Dr{\'e}04]{Drezet}
Jean-Marc Dr{\'e}zet.
\newblock {Luna's slice theorem and applications}.
\newblock In Jaroslaw~A. Wisniewski, editor, {\em {Algebraic group actions and
  quotients}}, pages 39--90. {Hindawi Publishing Corporation}, 2004.

\bibitem[ER21]{EdidinRydh}
Dan Edidin and David Rydh.
\newblock Canonical reduction of stabilizers for {A}rtin stacks with good
  moduli spaces.
\newblock {\em Duke Math. J.}, 170(5):827--880, 2021.

\bibitem[{Hal}20]{HL}
Daniel {Halpern-Leistner}.
\newblock {Derived $\Theta$-stratifications and the $D$-equivalence
  conjecture}.
\newblock {\em arXiv e-prints}, page arXiv:2010.01127, October 2020.

\bibitem[{Hek}21]{Hekking}
Jeroen {Hekking}.
\newblock {Graded algebras, projective spectra and blow-ups in derived
  algebraic geometry}.
\newblock {\em arXiv e-prints}, page arXiv:2106.01270, June 2021.

\bibitem[HL10]{HuyLehn}
Daniel Huybrechts and Manfred Lehn.
\newblock {\em The geometry of moduli spaces of sheaves}.
\newblock Cambridge Mathematical Library. Cambridge University Press,
  Cambridge, second edition, 2010.

\bibitem[Joy15]{JoyceDCrit}
Dominic Joyce.
\newblock A classical model for derived critical loci.
\newblock {\em J. Differential Geom.}, 101(2):289--367, 2015.

\bibitem[JS12]{JoyceSong}
Dominic Joyce and Yinan Song.
\newblock A theory of generalized {D}onaldson-{T}homas invariants.
\newblock {\em Mem. Amer. Math. Soc.}, 217(1020):iv+199, 2012.

\bibitem[Kir85]{Kirwan}
Frances~Clare Kirwan.
\newblock Partial desingularisations of quotients of nonsingular varieties and
  their {B}etti numbers.
\newblock {\em Ann. of Math. (2)}, 122(1):41--85, 1985.

\bibitem[KL13a]{cosection}
Young-Hoon Kiem and Jun Li.
\newblock Localizing virtual cycles by cosections.
\newblock {\em J. Amer. Math. Soc.}, 26(4):1025--1050, 2013.

\bibitem[KL13b]{KiemLi}
Young-Hoon Kiem and Jun Li.
\newblock A wall crossing formula of {D}onaldson-{T}homas invariants without
  {C}hern-{S}imons functional.
\newblock {\em Asian J. Math.}, 17(1):63--94, 2013.

\bibitem[KL18]{KiemLiCat}
Young-Hoon Kiem and Jun Li.
\newblock Critical virtual manifolds and perverse sheaves.
\newblock {\em J. Korean Math. Soc.}, 55(3):623--669, 2018.

\bibitem[KS10]{KontSoibel}
Maxim Kontsevich and Yan Soibelman.
\newblock Motivic {D}onaldson-{T}homas invariants: summary of results.
\newblock In {\em Mirror symmetry and tropical geometry}, volume 527 of {\em
  Contemp. Math.}, pages 55--89. Amer. Math. Soc., Providence, RI, 2010.

\bibitem[KS20]{KiemSavvasLoc}
Young-Hoon Kiem and Michail Savvas.
\newblock Localizing virtual structure sheaves for almost perfect obstruction
  theories.
\newblock {\em Forum of Mathematics, Sigma}, 8:e61, 2020.

\bibitem[KS21]{KiemSavvas}
Young-Hoon Kiem and Michail Savvas.
\newblock K-theoretic generalized {D}onaldson--{T}homas invariants.
\newblock {\em Int. Math. Res. Not. IMRN}, (24):19055--19090, 2021.

\bibitem[LT98]{LiTian}
Jun Li and Gang Tian.
\newblock Virtual moduli cycles and {G}romov-{W}itten invariants of algebraic
  varieties.
\newblock {\em J. Amer. Math. Soc.}, 11(1):119--174, 1998.

\bibitem[MFK94]{MFK}
David Mumford, James Fogarty, and Frances Kirwan.
\newblock {\em Geometric invariant theory}, volume~34 of {\em Ergebnisse der
  Mathematik und ihrer Grenzgebiete (2) [Results in Mathematics and Related
  Areas (2)]}.
\newblock Springer-Verlag, Berlin, third edition, 1994.

\bibitem[MNOP06a]{MNOP1}
Davesh Maulik, Nikita Nekrasov, Andrei Okounkov, and Rahul Pandharipande.
\newblock Gromov-{W}itten theory and {D}onaldson-{T}homas theory. {I}.
\newblock {\em Compos. Math.}, 142(5):1263--1285, 2006.

\bibitem[MNOP06b]{MNOP2}
Davesh Maulik, Nikita Nekrasov, Andrei Okounkov, and Rahul Pandharipande.
\newblock Gromov-{W}itten theory and {D}onaldson-{T}homas theory. {II}.
\newblock {\em Compos. Math.}, 142(5):1286--1304, 2006.

\bibitem[MT18]{MaulikToda}
Davesh Maulik and Yukinobu Toda.
\newblock Gopakumar-{V}afa invariants via vanishing cycles.
\newblock {\em Invent. Math.}, 213(3):1017--1097, 2018.

\bibitem[PT09]{PT1}
Rahul Pandharipande and Richard~P. Thomas.
\newblock Curve counting via stable pairs in the derived category.
\newblock {\em Invent. Math.}, 178(2):407--447, 2009.

\bibitem[PT10]{PT2}
Rahul Pandharipande and Richard~P. Thomas.
\newblock Stable pairs and {BPS} invariants.
\newblock {\em J. Amer. Math. Soc.}, 23(1):267--297, 2010.

\bibitem[PTVV13]{PTVV}
Tony Pantev, Bertrand To{\"e}n, Michel Vaqui{\'e}, and Gabriele Vezzosi.
\newblock Shifted symplectic structures.
\newblock {\em Publications math{\'e}matiques de l'IH{\'E}S}, 117(1):271--328,
  2013.

\bibitem[Rei89]{Reichstein}
Zinovy Reichstein.
\newblock Stability and equivariant maps.
\newblock {\em Inventiones mathematicae}, 96(2):349--383, Jun 1989.

\bibitem[{Sav}20]{Sav}
Michail {Savvas}.
\newblock {Generalized Donaldson-Thomas Invariants of Derived Objects via
  Kirwan Blowups}.
\newblock {\em arXiv e-prints}, page arXiv:2005.13768, May 2020.

\bibitem[Tho00]{Thomas}
Richard~P. Thomas.
\newblock A holomorphic {C}asson invariant for {C}alabi-{Y}au 3-folds, and
  bundles on {$K3$} fibrations.
\newblock {\em J. Differential Geom.}, 54(2):367--438, 2000.

\end{thebibliography}
\bibliographystyle{alpha}

\end{document}